\documentclass{sig-alternate-2013}

\newfont{\mycrnotice}{ptmr8t at 7pt}
\newfont{\myconfname}{ptmri8t at 7pt}
\let\crnotice\mycrnotice%
\let\confname\myconfname%

\usepackage{amsmath}    
\usepackage{graphics,epsfig,subfigure}    
\usepackage{verbatim}   
\usepackage{color}      
\usepackage{subfigure}  
\usepackage{amssymb}
\usepackage{wasysym}
\usepackage{epstopdf}
\usepackage{graphicx}
\usepackage{mathrsfs}

\newcommand{\prob}[1]{\text{Pr}\big\{#1\big\}}
\newcommand{\expect}[1]{\mathbb{E}\big\{#1\big\}}
\newcommand{\expectm}[1]{\mathbb{E}\{#1\}}

\newcommand{\bv}[1]{{\boldsymbol{#1} }}
\newcommand{\script}[1]{{{\cal{#1} }}}

\newtheorem{coro}{\textbf{Corollary}}
\newtheorem{lemma}{\textbf{Lemma}}

\newtheorem{theorem}{\textbf{Theorem}}

\newtheorem{assumption}{\textbf{Assumption}}
\newtheorem{definition}{\textbf{Definition}}

\newcommand{\olac}{$\mathtt{OLAC}$}
\newcommand{\olacb}{$\mathtt{OLAC2}$}

\allowdisplaybreaks

\begin{document}

\title{
The Power of Online Learning in Stochastic Network Optimization }
\author{ Longbo Huang$^\star$, Xin Liu$^\dagger$, Xiaohong Hao$^\star$\\
$^\star$\{longbohuang  haoxh10\}@tsinghua.edu.cn, IIIS, Tsinghua University\\
$^\dagger$liuxin@microsoft.com, Microsoft Research Asia
}

\markboth{Draft}{Huang}

\maketitle

\begin{abstract}

In this paper, we  investigate the power of online learning in  stochastic network optimization with unknown system statistics {\it a priori}. We are interested in understanding how information and learning can be efficiently incorporated into system control techniques, and what are the fundamental benefits of doing so. We propose two \emph{Online Learning-Aided Control} techniques, $\mathtt{OLAC}$ and $\mathtt{OLAC2}$, that explicitly utilize the past system information in current system control via a learning procedure called \emph{dual learning}. We prove strong performance guarantees of the proposed algorithms: $\mathtt{OLAC}$ and $\mathtt{OLAC2}$ achieve the near-optimal $[O(\epsilon), O([\log(1/\epsilon)]^2)]$ utility-delay tradeoff and $\mathtt{OLAC2}$ possesses an  $O(\epsilon^{-2/3})$ convergence time. Simulation results also confirm the superior performance of the proposed algorithms in practice. 
To the best of our knowledge, $\mathtt{OLAC}$ and $\mathtt{OLAC2}$ are  the first algorithms that simultaneously possess explicit near-optimal delay guarantee and sub-linear convergence time, and our attempt is the first to explicitly incorporate  online learning into stochastic network optimization and to demonstrate its power in both theory and practice. 
%
\end{abstract}

\section{Introduction} \label{section:intro}
Consider the following constrained network optimization problem: We are given a stochastic networked system with dynamic system states that have a stationary state distribution. At each state, an operation is implemented and a corresponding system cost occurs depending on the chosen actions. The objective is to minimize the expected cost given service/demand constraints.
Such a constrained optimization framework in stochastic systems is general and models many practical  application scenarios, such as in computer networks, smart grids, supply chain management, and transportation networks.
Due to this wide applicability, developing efficient control techniques for this framework in stochastic  systems has been one central question in network optimization.
However, solving this problem is very challenging and the main difficulty comes from the fact that the state distribution of the system is often unknown {\it a priori}. Moreover, known algorithms that handle this challenge either do not admit explicit delay guarantee, or suffer from a slow convergence speed. 
To address such a challenge and to overcome the previous limitations, in this paper, we investigate the value of online learning in optimal stochastic system control. 
We are interested in understanding how information and learning can be efficiently incorporated into system control techniques, and what are the fundamental benefits of doing so.

Specifically, we propose two \emph{Online Learning-Aided Control} techniques, $\mathtt{OLAC}$ and $\mathtt{OLAC2}$. The two new techniques are inspired by
the following key aspect of the recently developed Lyapunov technique (also known as Backpressure  \cite{neelynowbook}): 
the Lyapunov technique {converts the problem of finding optimal system control decisions into learning the optimal Lagrange multiplier of an underlying optimization problem in an incremental manner, i.e., at every time, the technique reacts only to the instantaneous system condition}.
This conversion allows the queues in the system to play the role of Lagrange multipliers, and the incremental nature eliminates the need for knowing the statistical information of the system.
Because of this attractive feature, the Lyapunov technique has received much attention in the literature. However, the main limitation of the  technique is that under Lyapunov algorithms, the queue size  in the system has to build up gradually, which  results in a large system delay. Specifically, in order to achieve an $O(\epsilon)$ close-to-optimal performance, the queue size has to be $\Theta(1/\epsilon)$, which is undesirable when $\epsilon$ is small.

In comparison, $\mathtt{OLAC}$ and $\mathtt{OLAC2}$ explicitly utilize the system information via a learning procedure called \emph{dual learning}, which effectively integrates the past system information into current system control by solving an \emph{empirical} optimal Lagrange multiplier. By doing so, $\mathtt{OLAC}$ and $\mathtt{OLAC2}$ convert the problem of optimal control into a combination of stochastic approximation and statistical learning. Note that this is a challenging task, because dual learning introduces another coupling effect in time to stochastic system algorithm performance analysis, which itself is already highly non-trivial due to the complicated interactions among different system components.


To the best of our knowledge, this is the first attempt to explicitly incorporate the power of online learning in stochastic network optimization. We address a general model of stochastic network optimization with unknown system statistics {\it a priori}. Specifically,
we make the following contributions in this paper:   

\begin{itemize}
\vspace{-.08in}
\item We propose two Online-Learning-Aided Control algorithms, $\mathtt{OLAC}$ and $\mathtt{OLAC2}$, which explicitly utilize the system information via online learning to integrate the past system information into current system control. $\mathtt{OLAC}$ and $\mathtt{OLAC2}$ apply to all scenarios where Backpressure applies. Moreover, they retain all the attractive features of Backpressure, such as having low computational complexities, and assuming zero  a priori statistical information. Therefore, they can be applied to large-scale dynamic network systems. Simulation results also demonstrate the empirical effectiveness of our proposed algorithms.


\vspace{-.08in}
\item We prove strong performance guarantees for the proposed algorithms, including near-optimality of the control policies, as well as sub-linear convergence time. Specifically, we show that $\mathtt{OLAC}$ and $\mathtt{OLAC2}$ achieve the near-optimal $[O(\epsilon), O([\log(1/\epsilon)]^2)]$ utility-delay tradeoff and $\mathtt{OLAC2}$ possesses a convergence time of $O(\epsilon^{-2/3})$, which significantly improves upon the $\Theta(1/\epsilon)$ convergence time of the Lyapunov technique. This is probably the first result that simultaneously achieves near-optimal utility-delay tradeoffs and sub-linear convergence time, and demonstrates the power of  online learning in stochastic network optimization. 

\vspace{-.08in}
\item We develop two analytical techniques, \emph{dual learning} and the \emph{augmented problem},  for  algorithm design and performance analysis. Dual learning allows us to connect dual subgradient update convergence to statistical convergence  of random system states, whereas the augmented problem  enables the interplay between Lyapunov drift analysis and duality in analyzing systems that apply time-inhomogeneous control policies.
These techniques may be applicable to solving other network control problems.

\vspace{-.06in}





\end{itemize}

The rest of the paper is organized as follows. In Section \ref{section:examples}, we discuss a few representative examples of stochastic network optimization in diverse application fields and the related works. We then explain how our results advance the state-of-the-art.  We set up our notations in Section \ref{section:notation}. We present the  system model and problem formulation in Section \ref{section:model}, and background information in Section \ref{section:review}. We present $\mathtt{OLAC}$ and $\mathtt{OLAC2}$ in Section \ref{section:algorithm}, and prove their optimality in Section \ref{section:performance}, and convergence in Section \ref{section:convergence}. Simulation results are presented in Section \ref{section:simulation}, followed by conclusions in Section \ref{section:conclusion}.

\section{Motivating Examples}\label{section:examples}

 In the following, we present a few representative examples of stochastic network optimization in diverse application fields and the related works, and explain how our main results can be applied.

\textbf{Wireless Networks}
Consider the following  simplified scheduling problem in cellular networks. A cellular base station (BS) is transmitting data to a mobile user. The channel (i.e., state) between the user and the BS is time varying,  and the cost (e.g., energy) and the transmission rate depend on the channel state.  The user has a certain arrival rate, and thus the constraint is that the  service rate provided by the BS to the user has to be larger than the arrival rate. The objective is to minimize the expected energy consumption, where the expectation is taken over channel distributions, subject to the rate constraint. This  example can be generalized in practice, e.g., to include multiple users, multiple hops, multiple transmission rates, various coding/modulation schemes,  multiple constraints, and different objectives, e.g., \cite{eryilmaz_qbsc_ton07}, \cite{rahulneelycognitive}. 

\textbf{Smart Grids} Consider the following demand response problem with renewable energy sources in a smart grid. The problem is to allocate renewable energy sources (e.g., solar or wind) to  flexible consumers (e.g., an EV to be  charged  or a dishwasher load to be finished). In this case, the renewable energy source
 provides energy according to a time-varying supply process (state). When the renewable energy source cannot generate enough energy to serve all customer load, it incurs  a cost to draw
energy from the regular power grid, which has a time-varying  price (state). The objective is to minimize the 
time average cost of using the regular grid (and hence results in the most efficient utilization of the renewable source). The constraint is that the average service rate has to be larger than the arrival rate of the consumer demand. Various related issues can also be considered here, including demand response \cite{opt-grid-tariff}, deferrable load scheduling \cite{neely-grid-10}, and energy storage management \cite{rahulneely-storage}.

\textbf{Supply Chain Management} Consider the following inventory control problem in supply chain management. A manufacturing plant  purchases
raw materials to assemble products and then sells the final products to customers. There are $K$ types of raw materials, each with a time-varying price (state) and $N$ types of final products, each with a time varying demand (state). At each time instance, the plant needs to decide whether to re-stock each type of raw materials and how to price each final product based on the current and future customer demands and material price. The objective is to maximize profit. In the case of large $K$ and $N$,  approximate dynamic programming solutions \cite{bertsekasdptbook} and  efficient Lyapunov optimization solutions have been proposed \cite{neelyhuang_assembly}. Other issues are studied in  general processing networks
\cite{dai-maxweight-spn}, \cite{jiang-spn}, with applications to semiconductor wafer fabrication facilities and assembly line control.

\textbf{Transportation Networks} Consider the traffic signal control problem in a transportation network. The operator controls each set of traffic signals at traffic intersections to regulate the traffic flow rate in the network.
Vehicles enter the network in a random fashion (state), with a constant average rate. They move in the network following  pre-specified average turn ratios. The objective is to stabilize the queue and allow vehicles to move as fast as possible. Such a system can be modeled as a ``store and forward"
queuing network. Feedback policies based on queue measurements have been extensively studied \cite{varaiya-bp-road}, as well as Backpressure based decentralized schemes \cite{le-signal-bp}.

\textbf{Solutions}
The above examples demonstrate the wide application scenarios of the stochastic optimization problem we consider in this paper. 
%
If  the distribution of the system state is known or if the system is static, many algorithms have been proposed using flow-based optimization techniques, e.g.,   \cite{low-flow-control}, \cite{layering-chiang07}, and the references therein. 
However, flow-based schemes typically do not explicitly characterize network delay and algorithm convergence time. 

For general stochastic settings, 
%
Backpressure algorithms address the challenge of unknown channel state distributions by gradually building up the queues and use them for optimal decision making, e.g., \cite{huangneelypricing-ton}, \cite{rahulneelycognitive}. 
However, Backpressure is known for its slow convergence and long delay,  because the queues have to be large enough for achieving near-optimal performance. Specifically, for achieving an $O(\epsilon)$ near-optimality, an $\Theta(1/\epsilon)$ queue size is required. 
There have been recent works trying to obtain improved utility-delay tradeoff for stochastic systems. For instance,  \cite{neelysuperfast}, \cite{huangneely_dr_tac}, and \cite{huang-lifo-ton} propose algorithms that can achieve the $[O(\epsilon), O([\log(1/\epsilon)]^2)]$ tradeoff. However, all the above algorithms require a convergence time of $\Theta(1/\epsilon)$. Moreover, they typically require additional system knowledge, e.g., system slack, for algorithm design, which  adds to the complexity of algorithm implementation. 

Our learning-based approach overcomes these limitations. In particular, we use an online learning-based approach to take advantage of the historic state information, which is ignored by Backpressure throughout the system control process. 
Our approach is based on a novel \emph{dual learning} idea, which computes an empirical Lagrange multiplier based on the empirical distribution of the system states. Then, we include  the learned Lagrange multiplier to the network queues for decision making. 
Two advantages manifest themselves in this learning-based approach. First, using the distribution information gradually learned in the system, we can significantly speed up the convergence to the optimal solution. Second, because of the ``virtual" value added to the queue, i.e., the empirical Lagrange multiplier,  the actual queue size is significantly smaller than that under Backpressure. Specifically, we develop two Online Learning-Aided Control techniques, $\mathtt{OLAC}$ and $\mathtt{OLAC2}$. We show that $\mathtt{OLAC}$ and $\mathtt{OLAC2}$ achieve the  $[O(\epsilon), O([\log(1/\epsilon)]^2)]$ tradeoff and $\mathtt{OLAC2}$ possesses a convergence time of $O(\epsilon^{-2/3})$, which significantly outperforms that of Backpressure.

\section{Notations}\label{section:notation}
$\mathbb{R}^n$ denotes the $n$-dimensional Euclidean space. $\mathbb{R}^n_+$ and $\mathbb{R}^n_-$ denote the non-negative and non-positive orthant. Bold symbols $\bv{x}=(x_1, ..., x_n)$ denote vectors in $\mathbb{R}^n$. The notion $w.p.1$ denotes ``with probability $1$.''  $\|\cdot\|$ denotes the Euclidean norm. For a sequence of variables $\{y(t)\}_{t=0}^{\infty}$, we also use $\overline{y}=\lim_{t\rightarrow\infty}\frac{1}{t}\sum_{\tau=0}^{t-1}\expect{y(\tau)}$ to denote its average (when exists). $\bv{x}\succeq\bv{y}$ means that $x_j\geq y_j$ for all $j$. 

\section{System Model and Problem Formulation}\label{section:model}
In this section, we specify the general network model. We consider a network controller that operates a network with the goal of minimizing the time average cost, subject to the queue stability constraint. The network is assumed to operate in slotted time, i.e., $t\in\{0,1,2,...\}$. We assume there are $r\geq1$ queues in the network (e.g., the amount of data to be transmitted in cellular networks or the amount of flexible jobs to be scheduled in a smart grid).

\subsection{Network State}\label{section:state}
In every slot $t$, we use $S(t)$ to denote the current network state, which indicates the current network parameters, such as a vector of conditions for each network link, or a collection of other relevant information about the current network channels and arrivals. 
We assume that $S(t)$ is i.i.d.~over time and takes $M$ different random network states denoted as $\script{S} = \{s_1, s_2, \ldots, s_M\}$. \footnote{The results in this paper can likely be generalized to systems with more general Markovian dynamics.} Let $\pi_{s_i}=\prob{S(t)=s_i}$ denote the probability of being in state $s_i$ at time $t$ and denote $\bv{\pi}=(\pi_{s_1}, ..., \pi_{s_M})$ the stationary distribution.
We assume that the network controller can observe $S(t)$ at the beginning of every slot $t$, but the $\pi_{s_i}$ probabilities are unknown.

\subsection{The Cost, Traffic, and Service}\label{subsection:costtrafficservice}
At each time $t$, after observing $S(t)=s_i$, the controller chooses an action $x(t)$ from a set $\script{X}^{(s_i)}$, i.e., $x(t)= x^{(s_i)}$ for some $x^{(s_i)}\in\script{X}^{(s_i)}$. The set $\script{X}^{(s_i)}$ is called the feasible action set for network state $s_i$ and is assumed to be time-invariant and compact for all $s_i\in\script{S}$.  The cost, traffic, and service generated by the chosen action $x(t)=x^{(s_i)}$ are as follows:
\begin{enumerate}
\vspace{-.06in}
\item[(a)] The chosen action has an associated cost given by the cost function $f(t)=f(s_i, x^{(s_i)}): \script{X}^{(s_i)}\mapsto \mathbb{R}_+$ (or $\script{X}^{(s_i)}\mapsto\mathbb{R}_-$ in reward maximization problems);
\vspace{-.06in}
\item[(b)] The amount of traffic generated by the action to queue $j$ is determined by the traffic function $A_j(t)=A_{j}(s_i, x^{(s_i)}): \script{X}^{(s_i)}\mapsto \mathbb{R}_{+}$, in units of packets;
\vspace{-.06in}
\item[(c)] The amount of service allocated to queue $j$ is given by the rate function $\mu_j(t)=\mu_{j}(s_i, x^{(s_i)}): \script{X}^{(s_i)}\mapsto \mathbb{R}_{+}$, in units of packets.
\vspace{-.06in}
 \end{enumerate}
Note that $A_j(t)$ includes both the exogenous arrivals from outside the network to queue $j$, and the endogenous arrivals from other queues, i.e., the transmitted packets from other queues, to queue $j$. We assume the functions $f(s_i, \cdot)$, $\mu_{j}(s_i, \cdot)$ and $A_{j}(s_i, \cdot)$ are time-invariant, their magnitudes are uniformly upper bounded by some constant $\delta_{\max}\in(0,\infty)$ for all $s_i$, $j$, and they are known to the network operator.


\subsection{Problem Formulation}
\label{section:queuenotation}
Let $\bv{q}(t)=(q_1(t), ..., q_r(t))^T\in\mathbb{R}^r_{+}$, $t=0, 1, 2, ...$ be the queue backlog vector  process of the network, in units of packets. We assume the following queueing dynamics: 
\begin{eqnarray}
q_j(t+1)=\max\big[q_j(t)-\mu_j(t)+A_j(t), 0\big], \quad\forall j,\label{eq:queuedynamic}
\end{eqnarray}
and $\bv{q}(0)=\bv{0}$. By using (\ref{eq:queuedynamic}), we assume that when a queue does not have enough packets to send, null packets are transmitted, so that the number of packets entering $q_j(t)$ is equal to $A_j(t)$.  In this paper, we adopt the following notion of queue stability \cite{neelynowbook}:
\vspace{-.06in}
\begin{eqnarray}
\overline{q}_{\text{av}}\triangleq
\limsup_{t\rightarrow\infty}\frac{1}{t}\sum_{\tau=0}^{t-1}\sum_{j=1}^{r}\expect{q_j(\tau)}<\infty.\label{eq:queuestable}
\end{eqnarray}
We use $\Pi$ to denote an action-choosing policy. Then,  we use $f^{\Pi}_{\text{av}}$ to denote the time average cost induced by $\Pi$, i.e., 
\vspace{-.06in}
\begin{eqnarray}
f^{\Pi}_{\text{av}}\triangleq
\limsup_{t\rightarrow\infty}\frac{1}{t}\sum_{\tau=0}^{t-1}\expect{f^{\Pi}(\tau)},\label{eq:timeavcost}
\end{eqnarray}
where $f^{\Pi}(\tau)$ is the cost incurred at time $\tau$ by policy $\Pi$. We call an action-choosing  policy \emph{feasible} if at every time slot $t$ it only chooses actions from the feasible action set $\script{X}^{(S(t))}$.  We then call a feasible action-choosing  policy under which (\ref{eq:queuestable}) holds a \emph{stable} policy, and use $f_{\text{av}}^*$ to denote the optimal time average cost over all stable policies.

In every slot, the network controller observes the current network state and chooses a control action, with the goal of minimizing the time average cost subject to network stability. This goal can be mathematically stated as:
\begin{eqnarray*}
\textbf{(P1)}\,\,\, \bv{\min: \,  f^{\Pi}_{\text{av}}, \,\, \text{s.t.}\,  (\ref{eq:queuestable})}.
\end{eqnarray*}
In the following, we call \textbf{(P1)} \emph{the stochastic problem}. It can be seen that the examples in Section \ref{section:examples} can all be modeled with the stochastic problem framework. This is the problem formulation  we focus on in this paper.



\vspace{-.02in}
\section{The Deterministic Problem and Backpressure}\label{section:review}
In this section, we first define the \emph{deterministic problem} and its dual problem, which will be important for designing our new control techniques and for our later analysis. We then review the Lyapunov technique   for solving the stochastic problem \textbf{(P1)}. To follow the convention, we will call it the Backpressure algorithm.

\subsection{The deterministic problem}
The \emph{deterministic problem} is  defined as follows \cite{huangneely_dr_tac}:
\begin{eqnarray}
\hspace{-.2in}&& \min: F(\bv{x}, \bv{\pi})\triangleq V\sum_{s_i}\pi_{s_i}f(s_i, x^{(s_i)})\label{eq:primal}\\
\hspace{-.2in}&&\quad\text{s.t.}\,\,\,H_j(\bv{x}, \bv{\pi})\\
\hspace{-.2in} && \qquad\qquad\triangleq\sum_{s_i}\pi_{s_i} [ A_j(s_i, x^{(s_i)})- \mu_j(s_i, x^{(s_i)})] \leq 0,\,\,\forall\, j,\nonumber\\
\hspace{-.2in}&& \qquad\quad  x^{(s_i)}\in \script{X}^{(s_i)}\quad \forall\, i=1, 2, ..., M. \nonumber
\end{eqnarray}
Here  the minimization is taken over $\bv{x}\in\prod_i\script{X}^{(s_i)}$, 
where $\bv{x}=(x^{(s_1)}, ..., x^{(s_M)})^T$, and $V\geq1$ is a positive constant introduced for later analysis. The dual problem of (\ref{eq:primal}) can be obtained as follows:
\begin{eqnarray}
\max:\,\,\, g(\bv{\gamma}),\quad \text{s.t.}\,\,\, \bv{\gamma}\succeq\bv{0},\label{eq:dualproblem}
\end{eqnarray}
where $g(\bv{\gamma})$ is the dual function and is defined as:
\vspace{-.06in}
\begin{eqnarray}
\hspace{-.3in}&&g(\bv{\gamma})=\inf_{x^{(s_i)}\in \script{X}^{(s_i)}}\sum_{s_i}\pi_{s_i}\bigg\{Vf(s_i, x^{(s_i)})\label{eq:dual_separable}\\
\hspace{-.3in}&&\qquad\qquad\qquad\qquad+\sum_j\gamma_j\big[A_j(s_i, x^{(s_i)})- \mu_j(s_i, x^{(s_i)})\big]\bigg\}.\nonumber
\end{eqnarray}
Here $\bv{\gamma}=(\gamma_1, ..., \gamma_r)^T$ is the \emph{Lagrange multiplier} of (\ref{eq:primal}). It is well known that $g(\bv{\gamma})$ in (\ref{eq:dual_separable}) is concave in the vector $\bv{\gamma}$ for all $\bv{\gamma}\in\mathbb{R}^r$, and hence the problem (\ref{eq:dualproblem}) can usually be solved efficiently, particularly when the cost functions and rate functions are separable over different network components  \cite{bertsekasoptbook}.

Below, we use $\bv{\gamma}^*=(\gamma^*_{1}, \gamma^*_{2}, ..., \gamma^*_{r})^T$ to denote an optimal solution of the problem (\ref{eq:dualproblem}). For notational convenience, we also use $g_0(\bv{\gamma})$ and $\bv{\gamma}_0^*$ to denote the dual function and an optimal dual solution for $V=1$. It can be seen that:
\vspace{-.06in}
\begin{eqnarray}
g(\bv{\gamma}) = Vg_0(\bv{\gamma}/V), \label{eq:dual-relationship}
\end{eqnarray}
which implies that $\bv{\gamma} = V\bv{\gamma}_0^*$ is an optimal solution of $g(\bv{\gamma})$. 
Note that $g_0(\bv{\gamma})$ is independent of $V$. For our later analysis, we also define:
\vspace{-.06in}
\begin{eqnarray}
\hspace{-.3in}&&g_{s_i}(\bv{\gamma})=\inf_{x^{(s_i)}\in \script{X}^{(s_i)}} \bigg\{Vf(s_i, x^{(s_i)})\label{eq:dual_single}\\
\hspace{-.3in}&&\qquad\qquad\qquad\qquad+\sum_j\gamma_j\big[A_j(s_i, x^{(s_i)})- \mu_j(s_i, x^{(s_i)})\big]\bigg\},\nonumber
\end{eqnarray}
to be the dual function when there is only a single state $s_i$. It is clear from equations (\ref{eq:dual_separable}) and (\ref{eq:dual_single}) that:
\begin{eqnarray}
g(\bv{\gamma}) = \sum_{s_i}\pi_{s_i}g_{s_i}(\bv{\gamma}). \label{eq:sum-dual}
\end{eqnarray}

\subsection{The Backpressure algorithm}
Among the many techniques developed for solving the stochastic problem, the Backpressure algorithm has received much attention because (i) it does not require any statistical information of the changing network conditions, (ii) it has low implementation complexity, and (iii) it has provable strong performance guarantees. 
The Backpressure algorithm works as follows  \cite{neelynowbook}. \footnote{A similar definition of Backpressure based on fluid model was also given in Section 4.8 of \cite{meyn-ccn}.}

\underline{\textsf{Backpressure:}} At every time slot $t$, observe the current network state $S(t)$ and the backlog $\bv{q}(t)$. If $S(t)=s_i$, choose $x^{(s_i)}\in\script{X}^{(s_i)}$ that solves the following:
\vspace{-.06in}
\begin{eqnarray}
\hspace{-.3in}\max: && -Vf(s_i, x)+\sum_{j=1}^{r}q_j(t)\big[\mu_j(s_i, x)-A_j(s_i, x)\big]\label{eq:QLAeq}\\
\text{s.t.} && x\in\script{X}^{(s_i)}.\nonumber\qquad\Diamond
\end{eqnarray}

In many problems, (\ref{eq:QLAeq}) can usually be decomposed into separate parts that are easier to solve, e.g., \cite{huangneelypricing-ton}, \cite{rahulneelycognitive}.
Also, when the network state process $S(t)$ is i.i.d., it has been shown in \cite{neelynowbook} that,
\begin{eqnarray}
f_{\text{av}}^{\mathtt{BP}}=f^*_{\text{av}}+O(1/V),\quad \overline{q}^{\mathtt{BP}}=O(V),\label{eq:qla_performance}
\end{eqnarray}
where $f_{\text{av}}^{\mathtt{BP}}$ and $\overline{q}^{\mathtt{BP}}$ are the expected average cost and the expected average network backlog size under Backpressure, respectively.
Note that the performance results in (\ref{eq:qla_performance}) hold under Backpressure with \emph{any} queueing discipline for choosing which packets to serve  and for any $V$.

Though being a low-complexity technique that possesses wide applicability, the delay performance and the convergence speed of Backpressure are not  satisfactory. Indeed, it is known that when Backpressure achieves a utility that is within $O(\epsilon)$ of the optimal, the average queueing delay is $\Theta(1/\epsilon)$. Although  techniques proposed in \cite{huangneely_dr_tac} \cite{huang-lifo-ton} are able to achieve an $O([\log(1/\epsilon)]^2)$ delay, it has been observed that the convergence time of these algorithms is  $\Theta(1/\epsilon)$ (this will also be proven in Section \ref{section:convergence}).
%

Moreover, we make the following observation of Backpressure: it discards all past information of the system states, i.e., $\{S(0), ..., S(t-1)\}$, and only reacts to the \emph{instantaneous} state. While such an \emph{incremental} manner is known to be able to accelerate the convergence of the algorithm compared to the ordinary subgradient methods \cite{bertsekasoptbook},  one interesting question to ask is whether such information can be utilized to construct better control techniques? Specifically, we are interested in understanding  \emph{how the  information can be incorporated into algorithm design and whether this incorporation enables the development of algorithms that possess better delay and convergence performance.}

In our next section, we present two learning-aided control techniques that perform learning through the  historic system state information.
As we will see, this learning step allows us to achieve a near-optimal system performance and a significant improvement in convergence speed.

\section{Online Learning-Aided Control}\label{section:algorithm}
In this section, we describe our online learning-aided system control idea and two novel control schemes, which we call \textsf{Online Learning-Aided Control} ($\mathtt{OLAC}$) and $\mathtt{OLAC2}$.

\subsection{Intuition}
Here we provide intuitions behind the two techniques. Both $\mathtt{OLAC}$ and $\mathtt{OLAC2}$ are motivated by the fact  that, under Backpressure algorithms, the queue vector plays the role of Lagrange multiplier \cite{huangneely_dr_tac}.  However, Backpressure's incremental nature ignores the possibility of utilizing the  information of the system dynamics for  ``accelerating'' the convergence of the control algorithm, and only relies on taking subgradient-type updates.  $\mathtt{OLAC}$ and $\mathtt{OLAC2}$ are designed to simultaneously take both subgradient-type updates and statistical learning into consideration.

\begin{figure}[ht]
\begin{center}
\includegraphics[width=3.2in, height=1.8in]{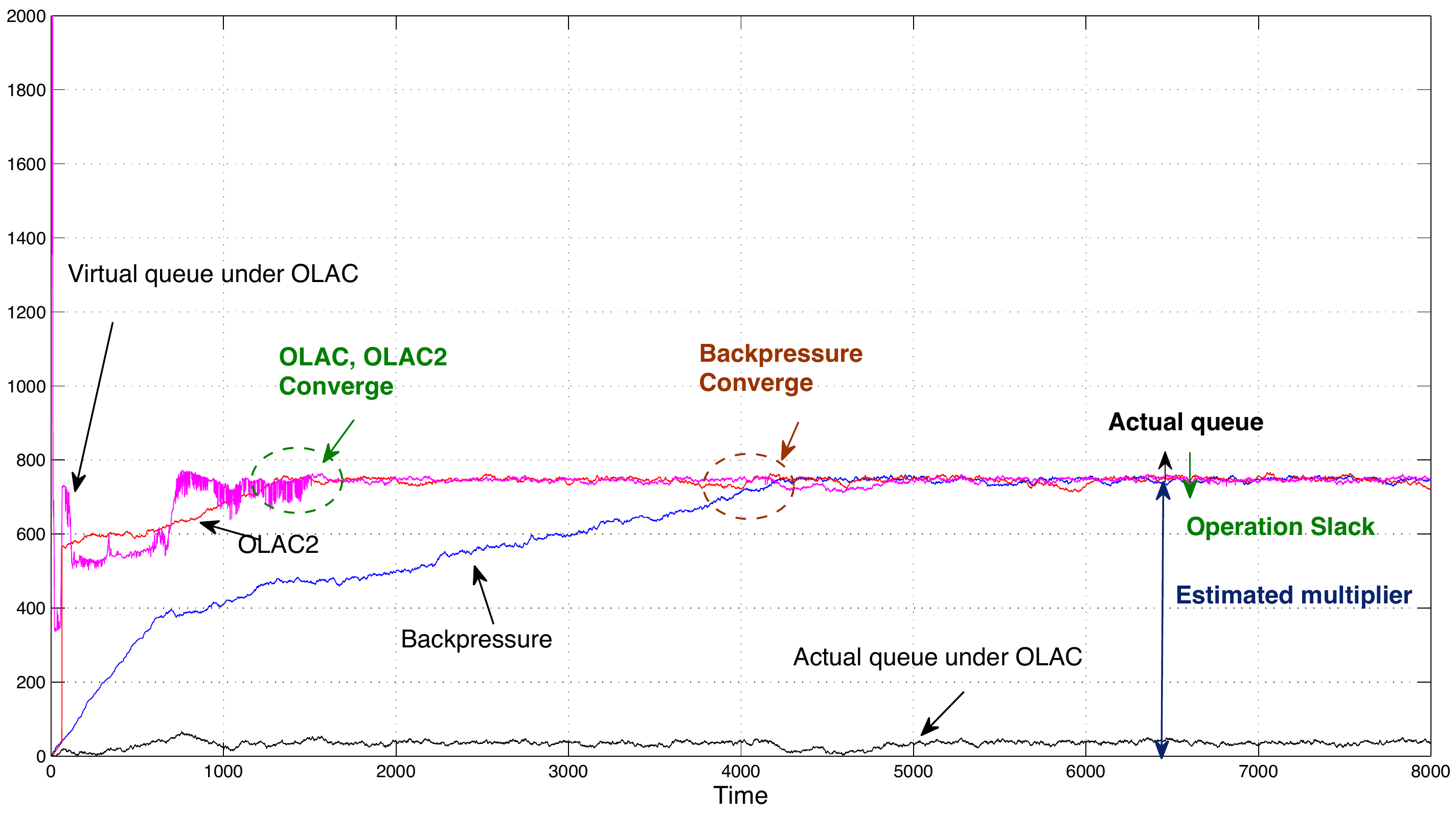}
\vspace{-.1in}
\caption{Convergence behavior of queue sizes under three algorithms, Backpressure, $\mathtt{OLAC}$ and $\mathtt{OLAC2}$. It can be seen clearly that the virtual queue under $\mathtt{OLAC}$ and the actual queue under $\mathtt{OLAC2}$ converge much faster to $\bv{\gamma}^*$ compared to the queue size under Backpressure. This clearly demonstrates the power of dual learning. }
\label{fig:intuition}
\end{center}
\vspace{-.2in}
\end{figure}
Fig. \ref{fig:intuition} best demonstrates our intuition and results.  As we can see from the figure, the queue vector under Backpressure is attracted to the fixed point $\bv{\gamma}^*$ after some time \cite{huangneely_dr_tac}. However, Backpressure only uses the queue value to track the value $\bv{\gamma}^*$. This results in undesired delay and convergence  performance. 
In $\mathtt{OLAC}$ and $\mathtt{OLAC2}$, we instead introduce an auxiliary variable $\bv{\beta}(t)$ to approach the value of $\bv{\gamma}^*$  by solving an ``empirical'' version of (\ref{eq:dual_separable}) and finding the empirical Lagrange multiplier (called \emph{dual learning}). The reason for performing this dual learning step is motivated by the fact that in many stochastic control problems, the optimal Lagrange multiplier $\bv{\gamma}^*$ is the unique key information  for finding the optimal control policies. 
Then, we feed  $\bv{\beta}(t)$  to the Backpressure controller for choosing the actions. Since $\bv{\beta}(t)$ eventually converges to $\bv{\gamma}^*$, we essentially do not require building up the queues for tracking $\bv{\gamma}^*$. Also, as we will see, at the beginning, dual learning provides a crude but fast learning of $\bv{\gamma}^*$. Hence, by doing so, we can significantly improve the convergence time of the system. 
%



\subsection{Dual Learning}
The dual learning step is a novel and critical component in our online-learning-aided control mechanisms.
Specifically, at every time $t$, the network operator maintains an empirical  distribution of the network states, denoted by $\bv{\pi}(t)=(\pi_{s_1}(t), ..., \pi_{s_M}(t))$, where $\pi_{s_i}(t)=N_{s_i}(t)/t$ and $N_{s_i}(t)$ is the number of slots where $S(t)=s_i$ in $\{0, 1, ..., t-1\}$. 
Note that  $\lim_{t\rightarrow\infty}\pi_{s_i}(t) = \pi_{s_i}$ $w.p.1$.

Although we adopt a simple learning mechanism for $\pi_{s_i}(t)$ here,  other approaches for learning the distribution can also be used. In addition, if some prior knowledge of the system exists, it can be incorporated into $\pi_{s_i}(0)$, which further speeds up the learning.
Other practical techniques in typical machine learning can also be included, such as using a uniform prior to avoid large fluctuation in the early stage of learning. 
Depending on the features of the approaches, similar performance and convergence results may also be obtained.

Using the empirical distribution, we  then define the following empirical dual  problem as follows:
\begin{eqnarray}
\hspace{-.3in}&&\max:\,\,g(\bv{\beta}, t)\triangleq \sum_{s_i}\pi_{s_i}(t)g_{s_i}(\bv{\beta}), \quad\text{s.t.}\,\,\,\, \bv{\beta}\succeq\bv{0}. \label{eq:primal-approx}
\end{eqnarray}
%
We denote $\bv{\beta}(t)=(\beta^*_{1}(t), \beta^*_{2}(t), ..., \beta^*_{r}(t))^T$ an optimal solution vector of (\ref{eq:primal-approx}), and we call this step of obtaining $\bv{\beta}(t)$ via (\ref{eq:primal-approx}) {\it dual learning}.  
 As we will see, $\bv{\beta}(t)$ plays a critical role in the proposed control schemes to significantly speed up the convergence to the optimal solution. On the other hand, by its definition, $\bv{\beta}(t)$ is time-varying and error-prone, which introduces significant challenges in the analysis of the proposed algorithms. Thus, novel  techniques are developed in Sections~\ref{section:performance} and \ref{section:convergence} for performance analysis.
 
In the following, we present two control techniques using the empirical Lagrange multiplier  $\bv{\beta}(t)$ obtained in dual learning for decision making.

%

\subsection{OLAC}
 In the first technique,  we use dual learning in each iteration of decision making. We also use a control parameter $\bv{\theta}>0$ to manage the queue deviation (Its function and value will be specified later).
 
\underline{\textsf{Online Learning-Aided Control} ($\mathtt{OLAC}$):} At every time slot $t$, do:
\begin{itemize}
\vspace{-.06in}
\item (Dual learning) Obtain $\bv{\beta}(t)$ by solving (\ref{eq:primal-approx}), 
\item (Action selection) Observe the system backlog $\bv{q}(t)$, and define the \emph{effective} backlog $\bv{Q}(t)$ with:
 \begin{eqnarray}
Q_j(t) = q_j(t)+\beta_j(t)-\theta_j,\quad \forall\,\, j. \label{eq:effective-LM}
\end{eqnarray}
Observe the current network state $S(t)$. If $S(t)=s_i$, choose $x^{(s_i)}\in\script{X}^{(s_i)}$ that solves the following:
\vspace{-.06in}
\begin{eqnarray}
\hspace{-.6in}&&\max: \,\, -Vf(s_i, x)+\sum_{j=1}^{r}Q_j(t)\big[\mu_j(s_i, x)-A_j(s_i, x)\big]\label{eq:lac-eq}\\
\hspace{-.6in}&&\,\,\text{s.t.} \qquad x\in\script{X}^{(s_i)}.\nonumber
\end{eqnarray}
\item (Queueing update) Update all the queues with the arrival and service rates under the chosen actions according to (\ref{eq:queuedynamic}).  $\Diamond$
\end{itemize}

We see that the $\mathtt{OLAC}$ algorithm has both similarities and differences compared  to the Backpressure algorithm. On one hand, $\mathtt{OLAC}$ retains all the advantages of Backpressure, i.e., it does not require any statistical information of the system and it retains the low-complexity feature.
On the other hand, one should  note that $\mathtt{OLAC}$ has an important component that is not possessed by Backpressure, i.e., dual learning through empirical Lagrange multiplier. This step fundamentally changes the algorithm. In particular, $\bv{\beta}(t)$ takes advantage of the empirical information to speed up the convergence of $\bv{Q}(t)$ to the optimal Lagrange multiplier. In addition, it also serves the function of a ``virtual" queue vector, and thus reduces the actual queue size of $\bv{q}(t)$ (See Fig. \ref{fig:intuition}). 

As we will see, $\mathtt{OLAC}$ achieves a near-optimal utility-delay tradeoff and has a significantly faster learning speed. The dual learning method also connects subgradient-type update analysis to the convergence properties of the system state distribution, allowing us to leverage useful results in the learning literature for performance analysis in stochastic network optimization.

A few comments on the control parameter $\bv{\theta}>0$ are in order. As we will see later, the effective backlog $\bv{Q}(t)$ will  eventually converge to $\bv{\gamma}^*$ (See the right side of Fig. \ref{fig:intuition}). 
Since $\bv{\beta}(t)$ also converges to $\bv{\gamma}^*$, this implies that $\bv{\beta}(t)+\bv{q}(t)$ will eventually 
be larger than $\bv{\gamma}^*$. Thus, if we were to use only $\bv{\beta}(t)+\bv{q}(t)$ as the weight for the Backpressure controller, the system will always ``think'' that it is in a congested stage and operate at an ``over-provisioned'' mode, which will lead to utility loss. The introduction of $\bv{\theta}$ is  trying to solve this problem. After subtracting $\bv{\theta}$, the effective backlog $\bv{Q}(t)$ can also go below $\bv{\gamma}^*$, which allows the Backpressure controller to choose ``under-provisioned'' actions and to achieve the desired performance through  time-sharing over-provisioned and under-provisioned actions \cite{neelyenergydelay}. We will also show in Section~\ref{section:performance} that it suffices to choose $\bv{\theta}=\Theta([\log(1/\epsilon)]^2)$ to guarantee an $O(\epsilon)$ close-to-optimal utility performance (Here $\epsilon=1/V$).  

\subsection{OLAC2}

We now present the second algorithm, which we call \textsf{Online Learning-Aided Control $2$  ($\mathtt{OLAC2}$)}. $\mathtt{OLAC2}$ combines the LIFO-Backpressure algorithm \cite{huang-lifo-ton} with dual learning, and contains a  backlog adjustment step. 
%

\underline{\textsf{Online Learning-Aided Control $2$ ($\mathtt{OLAC2}$):}} Choose a time $T_l=V^c$ for some $c\in[0, 1)$. At every time slot $t$, do:
\begin{itemize}
\vspace{-.06in}
\item (Action selection) Observe the current network state $S(t)$ and queue backlog $\bv{q}(t)$. If $S(t)=s_i$, choose $x^{(s_i)}\in\script{X}^{(s_i)}$ that solves the following:
\vspace{-.06in}
\begin{eqnarray}
\hspace{-.85in}&&\max: \,\, -Vf(s_i, x)+\sum_{j=1}^{r}q_j(t)\big[\mu_j(s_i, x)-A_j(s_i, x)\big] \label{eq:lac2-eq}\\
\hspace{-.85in}&&\,\,\text{s.t.} \qquad x\in\script{X}^{(s_i)}.\nonumber
\end{eqnarray}
\item (Queueing update) Update all the queues with the arrival and service rates under the chosen actions according to (\ref{eq:queuedynamic}) with the Last-In-First-Out (LIFO) discipline.
\item (Dual learning and backlog adjustment) Only at $t=T_l$, solve (\ref{eq:primal-approx}) with $\bv{\pi}(T_l-1)$ and obtain the empirical Lagrange multiplier $\tilde{\bv{\beta}}$. Then, make $\bv{q}(T_l)=\tilde{\bv{\beta}}$ by dropping packets if $q_j(T_l-1)>\tilde{\beta}_j$ or adding null packets if $q_j(T_l-1)<\tilde{\beta}_j$. $\Diamond$
\end{itemize}
Note that although there exists a packet dropping step in dual learning and backlog adjustment, packet dropping rarely happens. 
The intuition is that at $t=T_l =V^c$, we have with high probability that $\tilde{\beta} = O(V)$. On the other hand, since the queue size increment is always $\Theta(1)$, we only have $\bv{q}(T_l-1)=O(V^c)$. 
The key difference between $\mathtt{OLAC2}$ and the LIFO-Backpressure algorithm developed in \cite{huang-lifo-ton} is  that $\mathtt{OLAC2}$ contains a dual learning phase and adjusts its backlog condition at time $t=T_l$. The intuition here is that $\tilde{\bv{\beta}}$ computed at $t=T_l$ will give us good enough learning of the system, which provides better estimation of the true optimal Lagrangian multiplier than the queue value obtained via subgradient-type updates under LIFO-Backpressure.
As we will see in Section \ref{section:convergence} that, these two steps are critical for $\mathtt{OLAC2}$, as they  enable the  achievement of  a superior convergence performance. 

Note that the LIFO discipline is important to  $\mathtt{OLAC2}$. This is so because if $\mathtt{OLAC2}$ does not utilize any auxiliary process to ``substitute'' the true backlog. 
Thus, the queue will eventually build up, and its average value will still be  $\bv{\gamma}^*=\Theta(V)$. Hence, packets can experience large delay. 
However, LIFO scheduling ensures that most of the packets are not affected by such network congestion and can go through the system with low delay.

\subsection{Performance Guarantee}
We summarize the proven properties of $\mathtt{OLAC}$ and $\mathtt{OLAC2}$  here while leaving the technical details to next two sections.

 \olac: $\mathtt{OLAC}$ achieves the $[O(\epsilon), O([\log(1/\epsilon)]^2)]$ utility-delay tradeoff for general stochastic optimization problems. Compared to other Backpressure-based algorithms,  $\mathtt{OLAC}$ is easier to implement and shows much better convergence performance (see Figure~\ref{fig:intuition}). The convergence analysis  of  $\mathtt{OLAC}$ is much more difficult because of the time-varying nature of $\bv{\beta}(t)$, and thus is left for future work.

 \olacb: $\mathtt{OLAC2}$ achieves the $[O(\epsilon), O([\log(1/\epsilon)]^2)]$ utility-delay tradeoff for general stochastic optimization problems. In addition, it requires a convergence time of only $O(\epsilon^{-c} + \epsilon^{c/2-1}\log(1/\epsilon))$ with high probability. Therefore,  by selecting $c=2/3$ in \olacb, we  see that with very high probability, the system under $\mathtt{OLAC2}$  will enter the near-optimal state in only $O(\epsilon^{-2/3}\log(1/\epsilon))$ time! This is in high contrast to the Backpressure algorithm, whose convergence time is $\Theta(1/\epsilon)$.  

\section{Performance Analysis}\label{section:convergence} \label{section:performance}
In this section, we first present some preliminaries needed for our analysis. Then, we
present the detailed performance  results for $\mathtt{OLAC}$ and $\mathtt{OLAC2}$.

\subsection{The augmented problem and preliminaries}
Due to the introduction of $\bv{\beta}(t)$ in decision making, the performance of $\mathtt{OLAC}$ and $\mathtt{OLAC2}$ cannot be obtained by directly applying the typical Lyapunov analysis. To overcome this obstacle, we introduce the following \emph{augmented problem} and carry out our analysis based on it. Specifically, we define:
\begin{eqnarray}
\hspace{-.2in}&&\min: F(\bv{x}) + \sum_j (\beta_j(t) -\theta_j) H_j(\bv{x})
\label{eq:primal-aug}\\
\hspace{-.2in}&&\quad\text{s.t.}\,\,\,H_j(\bv{x})\leq 0, \,\,\forall\, j,\nonumber\\
\hspace{-.2in}&&\qquad\quad x^{(s_i)}\in \script{X}^{(s_i)}\quad \forall\, i=1, 2, ..., M. \nonumber
\end{eqnarray}
Let $\tilde{g}(\bv{\gamma})$ be the dual function of (\ref{eq:primal-aug}), i.e.,
\vspace{-.06in}
\begin{eqnarray}
\hspace{-.3in}&&\tilde{g}(\bv{\gamma})=\inf_{x^{(s_i)}\in \script{X}^{(s_i)}}\sum_{s_i}\pi_{s_i}\bigg\{Vf(s_i, x^{(s_i)})\label{eq:dual_separable-aug}\\
\hspace{-.3in}&& \qquad+\sum_j[\gamma_j + \beta_j(t) - \theta_j]\big[A_j(s_i, x^{(s_i)})- \mu_j(s_i, x^{(s_i)})\big]\bigg\}.\nonumber
\end{eqnarray}
It is interesting to notice the similarity between (\ref{eq:dual_separable-aug}) and (\ref{eq:lac-eq}) (note that $Q_j(t) = q_j(t)+\beta_j(t) -\theta_j$).
Using the definition of $g(\bv{\gamma})$, we have:
\begin{eqnarray}
\tilde{g}(\bv{\gamma}) = g(\bv{\gamma}+\bv{\beta}(t)-\bv{\theta}). \label{eq:g-tilde-eq}
\end{eqnarray}

In the following, we state the assumptions we make throughout the paper. 
These assumptions are mild and can  typically be satisfied in network optimization problems.
\vspace{-.06in}
\begin{assumption}\label{assumption:bdd-LM}
There exists a constant $\epsilon_s=\Theta(1)>0$ such that for any valid state distribution $\bv{\pi}' = (\pi'_{s_1}, ..., \pi'_{s_M})$ with $\|\bv{\pi}' - \bv{\pi} \|\leq \epsilon_s$, there exists a set of actions $\{x^{(s_i)}_k\}_{i=1,..., M}^{k=1,2, ..., \infty}$ with $x^{(s_i)}_k\in\script{X}^{(s_i)}$ and some variables $\vartheta^{(s_i)}_k\geq0$ for all $s_i$ and $k$ with $\sum_k\vartheta^{(s_i)}_k=1$ for all $s_i$ (possibly depending on $\bv{\pi}'$), such that:
\begin{eqnarray}
\sum_{s_i}\pi_{s_i}\big\{\sum_k\vartheta^{(s_i)}_k[A_{j}(s_i, x^{(s_i)}_k)-\mu_{j}(s_i, x^{(s_i)}_k)]\big\}\leq -\eta_0, \label{eq:slackness}
\end{eqnarray}
where $\eta_0=\Theta(1)>0$ is independent of $\bv{\pi}'$. $\Diamond$
\end{assumption}
\begin{assumption}\label{assumption:equal}
There exists a set of  actions $\{x^{(s_i)*}_k\}_{i=1,..., M}^{k=1, ..., \infty}$ with $x^{(s_i)*}_k\in\script{X}^{(s_i)}$ and some variables $\vartheta^{(s_i)*}_k\geq0$ for all $s_i$ and $k$ with $\sum_k\vartheta^{(s_i)*}_k=1$ for all $s_i$, such that:
\begin{eqnarray}
\hspace{-.45in}&&\qquad\qquad\qquad\qquad\,\,\,  \sum_{s_i}\pi_{s_i}\sum_k\vartheta^{(s_i)*}_kf(s_i, x^{(s_i)*}_k)  = f_{\text{av}}^*,  \\
\hspace{-.45in}&&\sum_{s_i}\pi_{s_i}\big\{\sum_k\vartheta^{(s_i)*}_k[A_{j}(s_i, x^{(s_i)*}_k)-\mu_{j}(s_i, x^{(s_i)*}_k)]\big\} = 0. \,\,\Diamond\nonumber
\end{eqnarray}
\end{assumption}
\begin{assumption}\label{assumption:unique}
$\bv{\gamma}_0^*$ is the  unique optimal solution of $g_0(\bv{\gamma})$ in $\mathbb{R}^r$. $\Diamond$
\vspace{-.06in}
\end{assumption}

Some remarks are in order. Notice that by having $\eta_0>0$ in Assumption \ref{assumption:bdd-LM}, we assume that there exists at least one control policy under which  
the resulting system arrival rate vector is strictly smaller than the resulting service rate vector (or the arrival rate vector is strictly inside the capacity region if it is exogenous). This is known as the ``slack'' condition, and is commonly made in the literature with $\epsilon_s=0$, e.g., \cite{ying_wmshortest_infocom09},  \cite{buisrikant_infocom09} (it is always necessary to have $\eta_0\geq0$ for system stability \cite{neelynowbook}). 
%
Here with $\epsilon_s>0$, we assume in addition that when two systems are relatively ``similar'' to each other, they can both be stabilized by some randomized control policy (may be different) that results in the same slack. 
 Assumption \ref{assumption:equal} is also commonly satisfied by most network optimization problems, especially when the cost $f(s_i, x^{(s_i)})$ increases with the increment of the services rate $\mu_j(s_i, x^{(s_i)})$. 
Finally, Assumption \ref{assumption:unique} holds for many network utility optimization problems, especially when the corresponding cost functions are strictly convex, e.g., \cite{eryilmaz_qbsc_ton07} and \cite{huangneely_dr_tac}.


Under these assumptions, our first lemma shows that the magnitude of $\bv{\beta}(t)$ quickly becomes bounded as time goes on.

\begin{lemma}\label{lemma:beta-conv}
There exists an $O(1)$ time $T_{\epsilon_s}<\infty$, such that with probability $1$, for all $t\geq T_{\epsilon_s}$,
\begin{eqnarray}
\hspace{-.4in}&&\sum_j\beta_j(t) \leq \xi\triangleq \frac{Vf_{\max}}{\eta_0}. \quad\Diamond\nonumber
\end{eqnarray}
\end{lemma}
\begin{proof}
See Appendix A.
\end{proof}

With the above bound, we have the following corollary, which shows  the convergence of $g(\bv{\beta}, t)$ to $g(\bv{\beta})$.
\begin{coro}\label{corollary:uniform-conv}
With probability $1$, for all $t\geq T_{\epsilon_s}$ (here $T_{\epsilon_s}$ is defined in Lemma \ref{lemma:beta-conv}), the function $g(\bv{\beta}, t)$ satisfies:
\begin{eqnarray}
\hspace{-.4in}&&|g(\bv{\beta}(t), t) -  g(\bv{\beta}(t))| \leq  \max_{s_i}|\delta_{s_i}(t)| M(Vf_{\max}+r\xi B). \nonumber
\end{eqnarray}
Here $\xi =   \frac{Vf_{\max}}{\eta_0}$ is defined in Lemma \ref{lemma:beta-conv},  $\delta_{s_i}(t) \triangleq \pi_{s_i} - \pi_{s_i}(t)$ is the estimation error of the empirical distribution for state $s_i$, and $M$ is the number of system states. $\Diamond$
\end{coro}
\begin{proof}
See Appendix B.
\end{proof}

With Lemma \ref{lemma:beta-conv} and Corollary \ref{corollary:uniform-conv}, we next show that  $\bv{\beta}(t)$ converges to $\bv{\gamma}^*$ with probability $1$. 
%
\begin{lemma}\label{lemma:conv}
$\lim_{t\rightarrow\infty} \bv{\beta}(t) = \bv{\gamma}^*$ $w.p.1$. $\Diamond$
\end{lemma}
\begin{proof}
See Appendix C.
\end{proof}
Notice that while Assumption \ref{assumption:unique} assumes that there is a unique optimal for (\ref{eq:primal}), it does not guarantee that $g(\bv{\beta}, t)$ will also only have a unique solution. Lemma \ref{lemma:conv} thus shows that although such a condition can appear, it is simply a transient phenomenon and $\bv{\beta}(t)$ will eventually converge to $\bv{\gamma}^*$.
According to Assumption \ref{assumption:unique}, (\ref{eq:g-tilde-eq}) implies that the unique maximizer of $\tilde{g}(\bv{\gamma})$, denoted by $\tilde{\bv{\gamma}}^*(t)$,  is given by $\tilde{\bv{\gamma}}^*(t)=\bv{\gamma}^* - \bv{\beta}(t) +\bv{\theta}$. Using Lemma \ref{lemma:conv}, we also see that:
\begin{eqnarray}
\lim_{t\rightarrow\infty}\tilde{\bv{\gamma}}^*(t)\rightarrow\bv{\theta}, \quad w.p.1. \label{eq:gamma-conv-theta}
\end{eqnarray}

In the following, we will carry out our analysis about  the performance of $\mathtt{OLAC}$ and $\mathtt{OLAC2}$ under a general system structure that commonly appears in practice. 
For our  analysis, it is also useful to define $B\triangleq\frac{r}{2}\delta_{\max}^2$. It can be seen that $\|\bv{\mu}(t) - \bv{A}(t)\|\leq B$ at all time.


\subsection{Performance of OLAC and OLAC2}
We now carry out our analysis for $\mathtt{OLAC}$ and $\mathtt{OLAC2}$ under a \emph{polyhedral} system structure. 
Intuitively speaking, this structure  appears in systems where the feasible control action sets are finite, in which case once an action becomes the minimizer of the dual function at $\bv{\gamma}$, it remains so at  all $\bv{\gamma}'$ in some neighborhood of $\bv{\gamma}$, resulting in a constant service and arrival rate difference (which determines the slope of $g(\bv{\gamma})$). 
The formal definition of the polyhedral structure is as follows \cite{huangneely_dr_tac}.
\vspace{-.06in}
\begin{definition}
A system is polyhedral with parameter $\rho>0$ if the dual function $g_0(\bv{\gamma})$ satisfies:
\begin{eqnarray}
g_0(\bv{\gamma}^*_0)\geq g_0(\bv{\gamma})+\rho\|\bv{\gamma}^*_0-\bv{\gamma}\|.\,\,\,\Diamond\label{eq:polyhedral}
\end{eqnarray}
\end{definition}

Using (\ref{eq:polyhedral}), we have:
\begin{eqnarray*}
Vg_0(\bv{\gamma}^*_0) \geq Vg_0(\bv{\gamma}/V)+V\rho\|\bv{\gamma}^*_0-\bv{\gamma}/V\|.
\end{eqnarray*}
Together with (\ref{eq:dual-relationship}) and the fact that $\bv{\gamma}^*=V\bv{\gamma}_0^*$, we see that the above implies:
\begin{eqnarray}
g(\bv{\gamma}^*) \geq g(\bv{\gamma})+\rho\|\bv{\gamma}^*-\bv{\gamma}\|,  \label{eq:polyhedral-v}
\end{eqnarray}
i.e., the function $g(\bv{\gamma})$ also satisfies the polyhedral condition (\ref{eq:polyhedral}) with the same parameter $\rho$. Moreover, we have:
\begin{eqnarray}
\tilde{g}(\tilde{\bv{\gamma}}^*(t)) &=& g(\tilde{\bv{\gamma}}^*(t)+\bv{\beta}(t)-\bv{\theta}) \nonumber\\
&\stackrel{(a)}{\geq}& g(\bv{\gamma}+\bv{\beta}(t)-\bv{\theta})+ \rho\|\tilde{\bv{\gamma}}^*(t)-\bv{\gamma}\|\nonumber\\
&=&\tilde{g}(\bv{\gamma})+\rho\|\tilde{\bv{\gamma}}^*(t)-\bv{\gamma}\|.\label{eq:exp-q-tilde}
\end{eqnarray}
Here (a) is because $\bv{\gamma}^* = \tilde{\bv{\gamma}}^*(t) + \bv{\beta}-\bv{\theta}$.
This shows that the function $\tilde{g}(\bv{\gamma})$ is also polyhedral.
It turns out that this polyhedral condition has a special attraction property, under which  the queue vector under $\mathtt{OLAC}$ and $\mathtt{OLAC2}$ will be exponentially attracted towards $\tilde{\bv{\gamma}}^*(t)$. This is shown in the  following important lemma:
\begin{lemma}\label{lemma:drift-g-tilde}
Suppose (\ref{eq:polyhedral}) holds. Then, there exist a constant $D_p\triangleq\frac{B-\eta^2}{2(\rho -\eta)} = \Theta(1)$ with $\eta<\rho$, and a finite time $T_0<\infty$, such that, with probability $1$,  for all $t\geq T_0$, if $\|\bv{q}(t)  -\tilde{\bv{\gamma}}^*(t)\|\geq D$,
\begin{eqnarray}
\hspace{-.0in}\expect{\|\bv{q}(t+1)  -\tilde{\bv{\gamma}}^*(t)\|\left|\right.\bv{q}(t)} \leq \|\bv{q}(t) - \tilde{\bv{\gamma}}^*(t)\|-\eta. \,\,\, \Diamond \label{eq:lac_exp_drift}
\end{eqnarray}
\end{lemma}
\begin{proof}
See Appendix D.
\end{proof}

Lemma \ref{lemma:drift-g-tilde} states that, after some finite time $T_0$, if the  queue vector $\bv{q}(t)$ is of distance  $D=\Theta(1)$ away  from $\tilde{\bv{\gamma}}^*(t)$, there will be an $\Theta(1)$ drift ``pushing'' the queue size towards $\tilde{\bv{\gamma}}^*(t)$. This intuitively explains why $\bv{q}(t)$ mostly stays close to a fixed point. Note that proof of Lemma \ref{lemma:drift-g-tilde} is based on the augmented problem (\ref{eq:primal-aug}). 
%
Below, we present the detailed performance results. Recall that all the results are obtained under Assumptions \ref{assumption:bdd-LM}, \ref{assumption:equal}, and \ref{assumption:unique}.  

\subsubsection{Performance of $\mathtt{OLAC}$}
We now analyze the performance of the $\mathtt{OLAC}$ algorithm. It is important to note that $\tilde{\bv{\gamma}}^*(t)$ is a function of $\bv{\beta}(t)$. Under $\mathtt{OLAC}$, however, the value of $\bv{\beta}(t)$ changes every time slot. Hence, analyzing the performance of $\mathtt{OLAC}$ is non-trivial.
Fortunately, using (\ref{eq:gamma-conv-theta}), we know that $\tilde{\bv{\gamma}}^*(t)$ eventually  converges to $\bv{\theta}$ with probability $1$. This allows us to prove the following theorem, which states that the average queue size under $\mathtt{OLAC}$ is mainly determined by the vector $\bv{\theta}$.
\begin{theorem}\label{theorem:q-poly}
Suppose (i) $g_0(\bv{\gamma})$ is polyhedral with $\rho=\Theta(1)>0$, and (ii)  $\bv{Q}(t)$ has a countable state space under $\mathtt{OLAC}$. Then, under $\mathtt{OLAC}$, we have $w.p.1$ that:
\begin{eqnarray}
\overline{q}_{\text{av}} = \sum_j\theta_j + O(1). \label{eq:avg-q}
\end{eqnarray}
\end{theorem}
\begin{proof}
See Appendix E.
\end{proof}

Theorem \ref{theorem:q-poly} shows that by using the decision making rule (\ref{eq:lac-eq}) with the effective backlog $\bv{Q}(t)$, one can  guarantee that the average queue size is roughly $\sum_j\theta_j$. Hence, by choosing $\theta_j=\Theta([\log(V)]^2)$, we recover the delay performance achieved in \cite{huangneely_dr_tac} and \cite{huang-lifo-ton}.

It is tempting to choose $\theta_j=0$, in which case one can indeed make the average queue size $\Theta(1)$. However,  as discussed before, the choice of $\bv{\theta}$ also affects the utility performance of $\mathtt{OLAC}$. Indeed, choosing a small $\bv{\theta}$  forces the system to run in an  \emph{over-provision} mode by always having an effectively backlog very close (or larger) to $\bv{\gamma}^*$. However,  in order to achieve a near-optimal utility performance, the control algorithm must be able to timeshare the over-provision mode and under-provision actions \cite{neelyenergydelay}. Therefore, it is necessary to use a larger $\bv{\theta}$ value.

In the following theorem, we show that an $O(1/V)$ close-to-optimal utility can be achieved as long as we choose $\theta_j=\Theta([\log(V)]^2)$ for all $j=1, ..., r$. Our analysis  is  different from the previous  Lyapunov analysis and will be useful for analyzing similar problems (Recall that $\epsilon=1/V$).
\begin{theorem}\label{theorem:cost}
Suppose the conditions in Theorem \ref{theorem:q-poly} hold.
Then,  with a sufficiently large $V$ and $\theta_j=[\log(V)]^2$ for all $j$, we have $w.p.1$ that:
\begin{eqnarray}
f^{\mathtt{OLAC}}_{\text{av}} = f_{\text{av}}^* + O(\frac{1}{V}). \label{eq:avg-u}
\end{eqnarray}
\end{theorem}
\begin{proof}
See Appendix F.
\end{proof}
Theorems \ref{theorem:q-poly} and \ref{theorem:cost} together show that $\mathtt{OLAC}$ achieves the $[O(1/V), O([\log(V)]^2)]$ utility-delay tradeoff for general stochastic optimization problems. Compared to the algorithms developed in \cite{huangneely_dr_tac} \cite{huang-lifo-ton}, which also achieve the same tradeoff, $\mathtt{OLAC}$ preserves using the First-In-First-Out (FIFO) queueing discipline and does not require any pre-learning phase. Thus, it is more suitable for practical implementations.

\subsubsection{Performance of $\mathtt{OLAC2}$}
We now present the performance results of $\mathtt{OLAC2}$. Since $\mathtt{OLAC2}$ is equivalent to LIFO-backpressure \cite{huang-lifo-ton} except for the learning step and a finite backlog adjustment at time $t=T_l$, its performance is the same as LIFO-Backpressure. \footnote{Note that $\mathtt{OLAC2}$ may need to discard packets in the backlog dropping step. However,  such an event happens with very small probability, as can be seen in the proof of Theorem \ref{theorem:lac2-convergence}.}

The following theorem from \cite{huang-lifo-ton} summarizes the results.
\begin{theorem}\label{theorem:lac2}
Suppose (i) $g_0(\bv{\gamma})$ is polyhedral with $\rho=\Theta(1)>0$, and (ii)  $\bv{q}(t)$ has a countable state space under $\mathtt{OLAC2}$. Then,  with a sufficiently large $V$,
\begin{enumerate}
\item[1)] The utility under $\mathtt{OLAC2}$ satisfies
$f^{\mathtt{OLAC2}}_{\text{av}} = f_{\text{av}}^* + O(\frac{1}{V})$. 

\item[2)] For any queue $j$ with a time average input rate $\lambda_j>0$,  there are a subset of packets from the arrivals that eventually depart the queue and have an average rate $\tilde{\lambda}_j$ that satisfies:
\begin{eqnarray}
\lambda_j \geq \tilde{\lambda}_j \geq \left[\lambda_j - O(\frac{1}{V^{\log(V)}}) \right]^+. \label{eq:rate-major}
\end{eqnarray}
Moreover, the average delay of these packets is $O(\frac{[\log(V)]^2}{\tilde{\lambda}_j})$.
\end{enumerate}
\end{theorem}
\begin{proof}
See \cite{huang-lifo-ton}.
\end{proof}
Theorem \ref{theorem:lac2} shows that if a queue $q_j(t)$ has an input rate  $\lambda_j=\Theta(1)$, then, under $\mathtt{OLAC2}$, almost all the packets going through $q_j(t)$ will experience only $O([\log(V)]^2)$ delay. Applying this argument to every queue in the network, we see that the average network delay is roughly $O(r[\log(V)]^2)$.

\section{Convergence Time Analysis}\label{section:convergence}
In this section, we study another important performance metric of our techniques, the convergence time. Convergence time measures how fast an algorithm reaches its steady-state. Hence, it is an important indicator of the robustness of the technique.
However, it turns out that due to the complex nature of the $\bv{\beta}(t)$ process   and the structure of the systems, the convergence properties of the algorithms are hard to analyze. Therefore, we mainly focus on analyzing the convergence time of  $\mathtt{OLAC2}$. 
We now give the formal definition of the convergence time of an algorithm. In the definition, we use $\bv{\gamma}(t)$ to denote the estimated Lagrange multiplier under the control schemes, e.g., $\bv{q}(t)$ under Backpressure and $\mathtt{OLAC2}$  and $\bv{\beta}(t)-\bv{\theta}+\bv{q}(t)$ under $\mathtt{OLAC}$.
\begin{definition}
Let $\zeta>0$ be a given constant.  The $\zeta$-convergence time of the control algorithm, denoted by $T_{\zeta}$, is the time it takes for the estimated Lagrange multiplier $\bv{\gamma}(t)$ to get to within $\zeta$ distance of $\bv{\gamma}^*$, i.e.,
\begin{eqnarray}
T_{\zeta}\triangleq\inf\{t\,|\, ||\bv{\gamma}(t)-\bv{\gamma}^*||\leq\zeta\}.\,\,\,\Diamond\label{eq:con-time}
\end{eqnarray}
\end{definition}
Note that our definition is different from the definition in \cite{li-convergence-13}, where  convergence time relates to how fast the time-average rates converge to the optimal values. In our case, the convergence time definition  (\ref{eq:con-time}) is motivated by the fact that Backpressure, $\mathtt{OLAC}$, and $\mathtt{OLAC2}$ all use functions of the queue vector to estimate $\bv{\gamma}^*$, which is the key for determining the optimal control actions. Hence, the faster the algorithm learns $\bv{\gamma}^*$, the faster the system enters the optimal operating zone.

In the following, we present the convergence results of $\mathtt{OLAC2}$.
As a benchmark comparison, we first study the convergence time of the Backpressure algorithm. Since both Backpressure and $\mathtt{OLAC2}$ use the actual queue size as the estimate of $\bv{\gamma}^*$, we will analyze the convergence time with $\bv{\gamma}(t)=\bv{q}(t)$.


\subsection{Convergence time of Backpressure}
We start by analyzing the convergence time of Backpressure. Our result is summarized in the following theorem.
\begin{theorem}\label{theorem:convergence-bp}
Suppose that $g_0(\bv{\gamma})$ is polyhedral with parameter $\rho=\Theta(1)>0$. Then,
\begin{eqnarray}
\big(\|\bv{q}(0)-\bv{\gamma}^*\| - D_p  \big)^+/B \leq \expect{T_{D_p}} \leq \|\bv{q}(0)-\bv{\gamma}^*\|/\eta.
\end{eqnarray}
Here $\eta\in[0, \rho)=\Theta(1)$ and  $D_p=\frac{B-\eta^2}{2(\rho -\eta)}=\Theta(1)$. $\Diamond$
\end{theorem}
\begin{proof}
See Appendix G.
\end{proof}
Since $\bv{\gamma}^*=\Theta(V)$, if  $\bv{q}(0)=\Theta(1)$, e.g., $\bv{q}(0)=\bv{0}$, Theorem \ref{theorem:convergence-bp} states that the expected convergence time of Backpressure is $\Theta(V)$.
This is consistent with the results that have been observed in previous works, e.g., \cite{neelynowbook} and \cite{huangneely_dr_tac}.


%

\subsection{Convergence time of OLAC2} 
We now study the convergence time of the  $\mathtt{OLAC2}$ scheme. 
Notice that under $\mathtt{OLAC2}$ the convergence of $\bv{q}(t)$ to $\bv{\gamma}^*$ depends on both $\tilde{\bv{\beta}}$ and the dynamics of  the queue vector $\bv{q}(t)$. 
The following theorem characterizes the convergence time of $\mathtt{OLAC2}$.
\begin{theorem}\label{theorem:lac2-convergence}
Suppose (i) $g_0(\bv{\gamma})$ is polyhedral with $\rho=\Theta(1)>0$, and (ii) $\bv{q}(t)$ has a countable state space under $\mathtt{OLAC2}$. Then,  with a sufficiently large $V$, we have  with probability of at least $1- \frac{M}{V^{4\log(V)}}$ that, under $\mathtt{OLAC2}$,
\begin{eqnarray}
\expect{T_{D_p}} =  O(V^{1-c/2}\log(V) + V^c). \quad\Diamond\label{eq:lac2-convergence}
\end{eqnarray}
\end{theorem}
 \begin{proof}
 See Appendix H.
 \end{proof}
Choosing $c= \frac{2}{3}$, we notice that with very high probability, the system under $\mathtt{OLAC2}$  will enter the near-optimal state in only $O(V^{2/3}\log(V))$ time! This is in  contrast to the Backpressure algorithm, whose convergence time is $\Theta(V)$ as shown in Theorem \ref{theorem:convergence-bp}.    

\section{Simulation}\label{section:simulation}
In this section, we provide simulation results for $\mathtt{OLAC}$ and $\mathtt{OLAC2}$ to demonstrate both the near-optimal performance and the the superior convergence rate. We consider a $2$-queue system depicted in Fig. \ref{fig:2q}. 
\begin{figure}[cht]
\begin{center}
\vspace{-.06in}
\includegraphics[width=2.0in, height=0.8in]{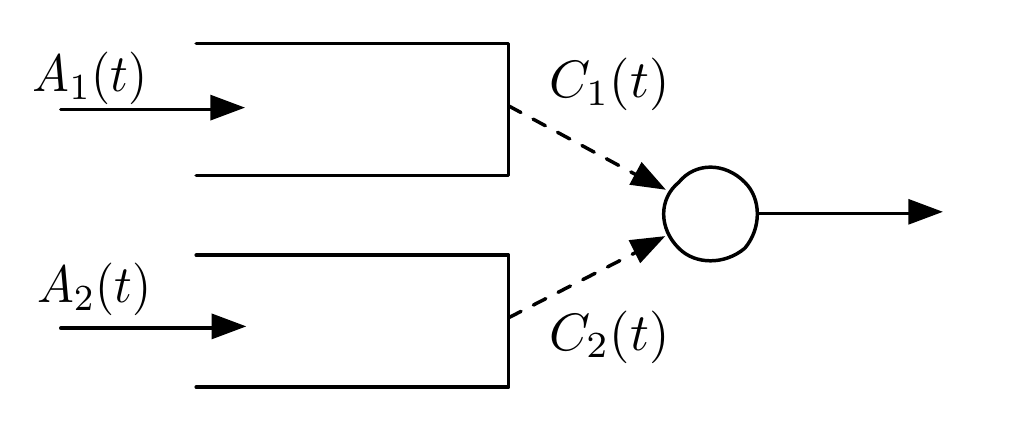}
\vspace{-.1in}
\caption{A $2$-queue system. In this system, each queue receives random arrivals. The server can only serve one queue at a time. }
\label{fig:2q}
\end{center}
\vspace{-.2in}
\end{figure}

Here $A_j(t)$ denotes the number of arriving packets to queue $j$ at time $t$. We assume that $A_j(t)$ is i.i.d. with either $2$ or $0$ with probabilities $p_j$ and $1-p_j$, where $p_1=0.3$ and $p_2=0.4$. In this case, we see that $\lambda_1=0.6$ and $\lambda_2=0.8$. We assume that each queue has a time-varying channel and denote its state by $C_j(t)$.  $C_j(t)$ takes value in the set $\script{C} = \{0, 2, 4, 6\}$. 
%
At each time, the server decides how much power to allocate to each queue. We denote $P_j(t)$ the power allocated to queue $j$ at time $t$. Then, the instantaneous service rate a queue obtains is given by:
\begin{eqnarray}
\mu_j(t) = \log(1+C_j(t) P_j(t)). 
\end{eqnarray} 
The feasible power allocation set is given by $\script{P}=\{ 0, 0.75$, $1.5$, $2.25$, $3\}$. 
The objective is to stabilize the system with minimum average power. 
It is important to note that, even though this is a simple setting, it is actually quite representative and models many problems in different contexts, e.g.,  a downlink system in wireless network, workload scheduling in a power system, inventory control system, and traffic light control. Moreover, it can be verified that Assumptions \ref{assumption:bdd-LM}, \ref{assumption:equal} and \ref{assumption:unique} all hold in this example.

We compare three algorithms, Backpressure, $\mathtt{OLAC}$ and $\mathtt{OLAC2}$. We also consider two different channel distributions. In the first distribution, $C_j(t)$ takes each value in $\script{C}$ with probability $0.25$, 
whereas in the second distribution, $C_j(t)$ takes values $0$ and $6$ with probability $0.1$ and takes values $2$ and $4$ with probability $0.4$. This is to mimic the Gaussian distribution. Note that in each case, we have a total of $16$ different channel state combinations. 

Fig. \ref{fig:per-uniform} and \ref{fig:per-unbalanced} show that $\mathtt{OLAC}$ and $\mathtt{OLAC2}$ significantly outperform the Backpressure in terms of delay. For example, in the uniform channel case, we can see from Fig. \ref{fig:per-uniform} that when $V=100$, the average power performance is indistinguishable under the three algorithms. Backpressure results in an average delay of $210$ slots while $\mathtt{OLAC}$ and $\mathtt{OLAC2}$ only incur an average delay about $20$ slots, which is \emph{$10$ times} smaller! Fig.  \ref{fig:per-unbalanced}  shows similar behavior of the algorithms under the unbalanced channel distribution. 
\begin{figure}[cht]
\begin{center}
\vspace{-.06in}
\includegraphics[width=1.6in, height=1.6in]{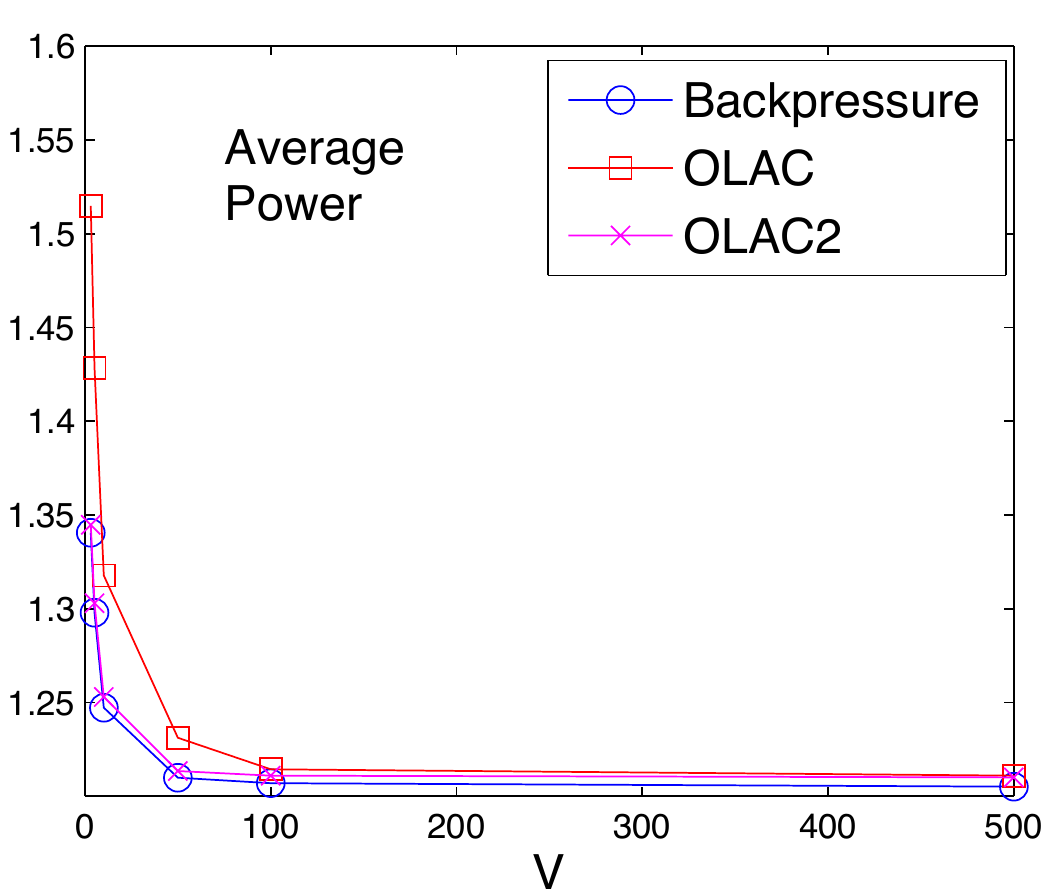}
\includegraphics[width=1.6in, height=1.6in]{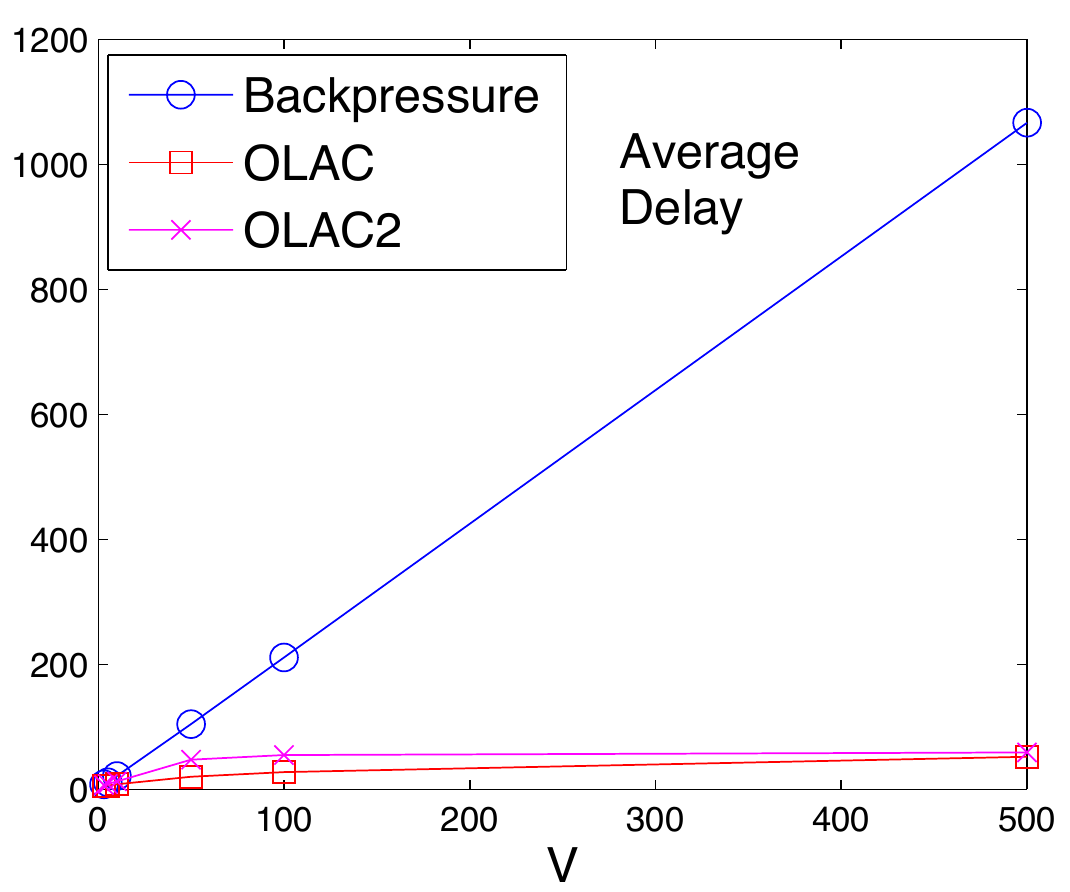}
\vspace{-.1in}
\caption{The power-delay performance of the algorithms under uniform channel distribution $[0.25, 0.25, 0.25, 0.25]$. We see that although all algorithms achieve similar near-minimum average power,  $\mathtt{OLAC}$ and $\mathtt{OLAC2}$ achieve this performance with much smaller delay.} 
\label{fig:per-uniform}
\end{center}
\vspace{-.2in}
\end{figure}

\begin{figure}[cht]
\begin{center}
\vspace{-.06in}
\includegraphics[width=1.6in, height=1.6in]{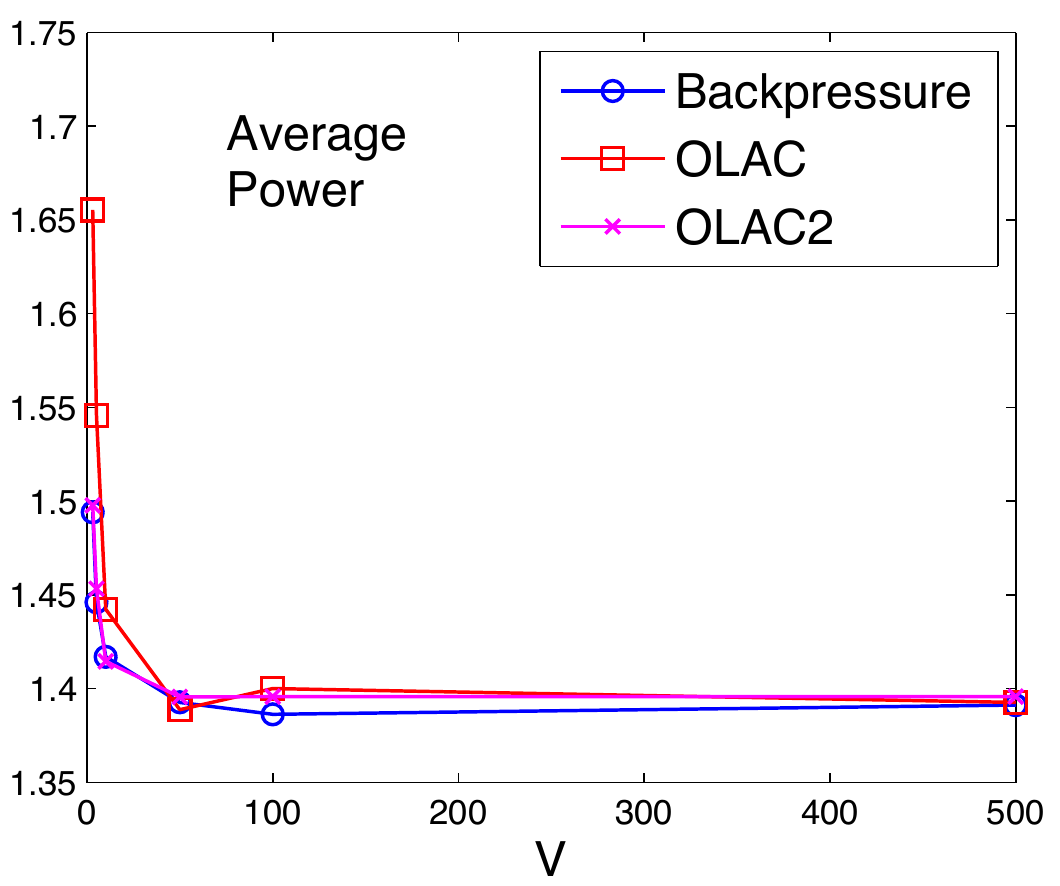}
\includegraphics[width=1.6in, height=1.6in]{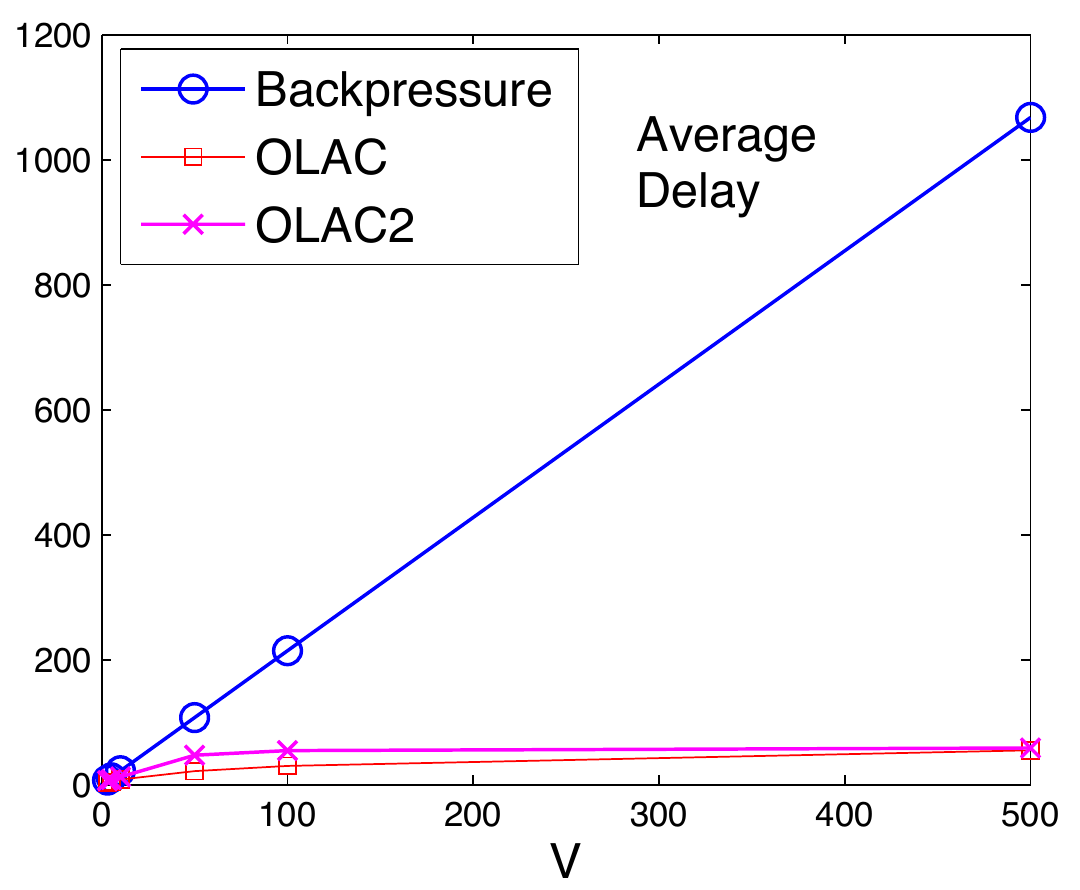}
\vspace{-.1in}
\caption{The power-delay performance of the algorithms under the unbalanced channel distribution $[0.1, 0.4, 0.4, 0.1]$. We see again that $\mathtt{OLAC}$ and $\mathtt{OLAC2}$ significantly outperform Backpressure in delay.} 
\label{fig:per-unbalanced}
\end{center}
\vspace{-.2in}
\end{figure}

Having seen the utility-delay tradeoff of the algorithms, we now look at the convergence rate of the algorithms. Fig. \ref{fig:intuition} shows the convergence behavior of the first queue under the three different algorithms under the uniform channel distribution with $V=500$. As we can see, the virtual queue size under $\mathtt{OLAC}$ and the actual queue size under $\mathtt{OLAC2}$ converge very quickly. Compared to Backpressure, the convergence time is reduced by about $2500$ timeslots. The behavior of the other queue is similar and hence we do not present the results. It is important to notice here that since $\mathtt{OLAC}$ and $\mathtt{OLAC2}$ converge very quickly, the early arrivals into the system will depart from the queue without much waiting time (see the queue process under $\mathtt{OLAC}$), whereas in Backpressure, early arrivals must wait until the algorithm converge to get serve and suffer from a long delay. 
Fig. \ref{fig:convergence-lac2} also shows the ``jump'' behavior of the sizes of the queues under $\mathtt{OLAC2}$ under the uniform channel distribution and $V=500$. 
It can be seen that the convergence time is significantly reduced via such dual learning (by about $2500$ timeslots). 
\begin{figure}[cht]
\begin{center}
\vspace{-.06in}
\includegraphics[width=3in, height=1.6in]{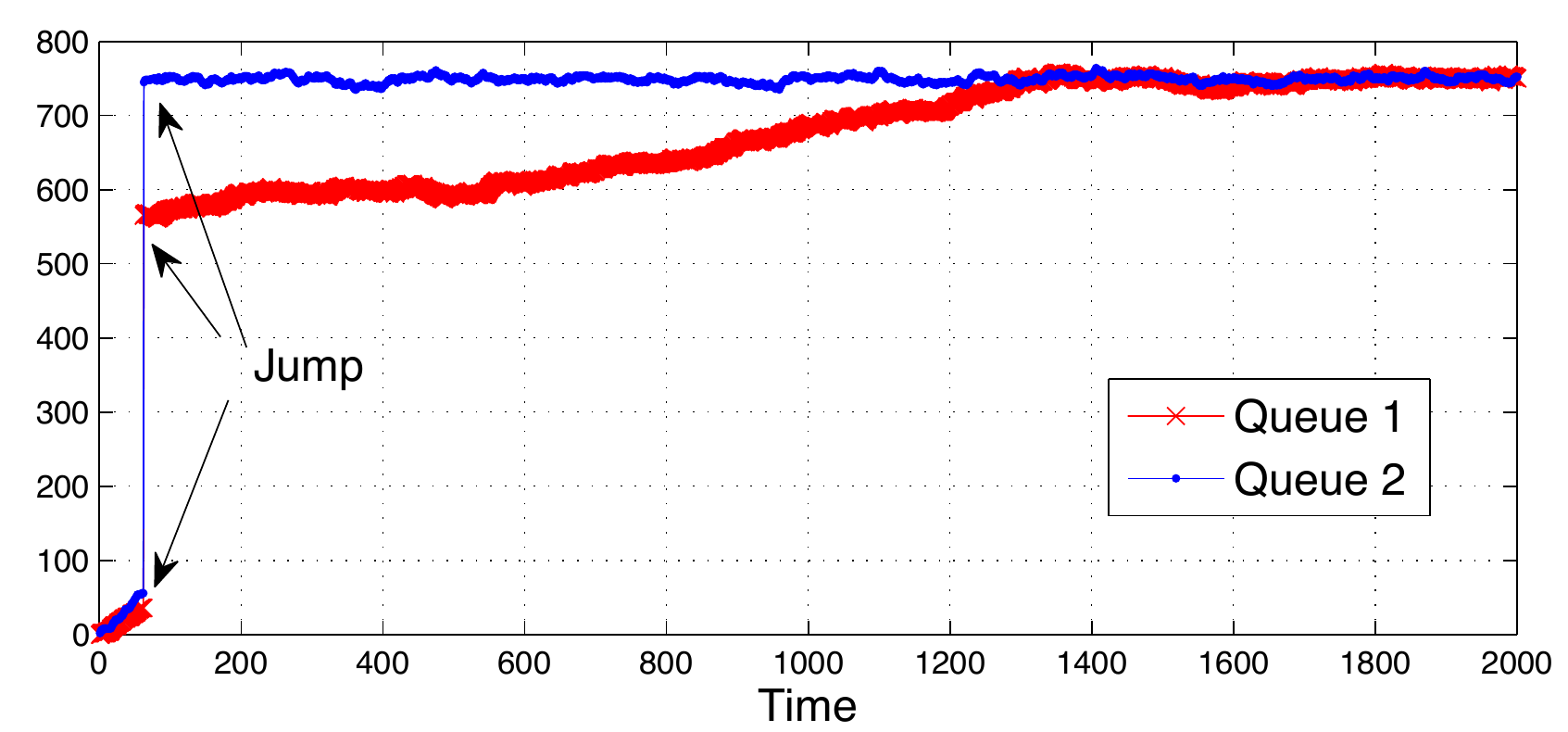}
\vspace{-.1in}
\caption{Convergence of the queues to the optimal Lagrange multiplier under $\mathtt{OLAC2}$ with uniform channel distribution and $V=500$. It can be seen that after only around $80$ slots, dual learning already generates a very good estimation of $\bv{\gamma}^*$. This significantly reduces the time for convergence. } 
\label{fig:convergence-lac2}
\end{center}
\vspace{-.2in}
\end{figure}




\section{Conclusion} \label{section:conclusion}
In this paper, we take a first step to investigate the power of online learning
in stochastic network optimization with unknown system
statistics {\it a priori}. We propose two learning-aided control techniques  $\mathtt{OLAC}$ and $\mathtt{OLAC2}$ for controlling stochastic queueing systems. The design of $\mathtt{OLAC}$ and $\mathtt{OLAC2}$ incorporates statistical learning into system control,  and provides  new ways for designing algorithms for achieving  near-optimal system performance  for general system utility maximization problems. Moreover,  incorporating the learning step significantly improves  the convergence speed of the control algorithms. Such insights are not only proven via novel analytical techniques, but also validated through numerical simulations. Our study demonstrates the promising power of online learning in stochastic network optimization. Much further investigation is desired. In particular, we would like to provide theoretical guarantees on the convergence speed on \olac, and more importantly, understand the fundamental limit of convergence in the online-learning-aided stochastic network optimization. We would like to further investigate practical application scenarios of the proposed schemes and address the corresponding challenges in real world applications.

\section{Acknowledgement}
This work was supported in part by the National Basic Research Program of China Grant 2011CBA00300,2011CBA00301, the National Natural Science Foundation of China Grant 61033001, 61361136003, 041302007, the China youth 1000-talent grant 610502003, and the MSRA-Tsinghua joint lab grant 041902004.
\bibliographystyle{unsrt}
\bibliography{mybib}

\section*{Appendix A -- Proof of Lemma \ref{lemma:beta-conv}}
We prove Lemma \ref{lemma:beta-conv} here.
\begin{proof} (Lemma \ref{lemma:beta-conv})
First we see that $\pi_{s_i}(t)$ converges to $\pi_{s_i}$ as $t$ goes to infinity with probability $1$ \cite{durrett_prob}. Thus, there exists a time $T_{\epsilon_s}<\infty$, such that $\|\bv{\pi}(t) - \bv{\pi} \|\leq \epsilon_{s}$ for all $t\geq T_{\epsilon_s}$ with probability $1$.
This implies that $\bv{\beta}(t)<\infty$ for all $t\geq T_{\epsilon_s}$ \cite{bertsekasoptbook}. 
Indeed, under Assumption \ref{assumption:bdd-LM},  we see that there exists an $\eta_0>0$ such that (\ref{eq:slackness}) holds. 

We now construct a \emph{fictitious} system, which is exactly the same as our system, except that we replace the network state distribution $\bv{\pi}$ by $\bv{\pi}(t)$. From  Assumption \ref{assumption:bdd-LM}, we know that for any $t\geq T_{\epsilon_s}$, $w.p.1$,  this system admits an optimal control policy that achieves the optimal cost (with the state distribution being $\bv{\pi}(t)$) and ensures network stability \cite{neelynowbook}. We denote  $f_{\text{av}}^*(\bv{\pi}(t))$ the optimal average cost of the fictitious system subject to stability. We see then $f_{\text{av}}^*(\bv{\pi}(t))\geq0$. 

Now we apply Theorem 1 in \cite{huangneely_qlamarkovian} to get that $g(\bv{\beta}(t), t)=f_{\text{av}}^*(\bv{\pi}(t))$. 
%
Therefore, for any $t\geq T_{\epsilon_s}$, one can plug the variables that result in  the slack in (\ref{eq:slackness}) into $g(\bv{\beta}, t)$ and obtain:
\begin{eqnarray*}
0&\leq& f_{\text{av}}^*(\bv{\pi}(t))\\
&\stackrel{(a)}{=}& g(\bv{\beta}(t), t) \\
&\stackrel{(b)}{\leq}& \sum_{s_i} \pi_{s_i}(t) [Vf_{\max} - \eta_0\sum_j\beta_{j}(t)].
\end{eqnarray*}
Here (a) follows from Theorem 1 in \cite{huangneely_qlamarkovian}, and (b) follows from the definition of $g(\bv{\beta}, t)$.
This shows that $w.p.1$,
\begin{eqnarray}
\sum_j\beta_{j}(t) \leq \xi\triangleq \frac{Vf_{\max}}{\eta_0}, \,\,\forall\,\, t\geq T_{\epsilon_s}.
\end{eqnarray}
This proves Lemma \ref{lemma:beta-conv}.
\end{proof}

\section*{Appendix B -- Proof of Corollary \ref{corollary:uniform-conv}}
We carry out our proof for Corollary \ref{corollary:uniform-conv} here.
\begin{proof} (Corollary \ref{corollary:uniform-conv})
From lemma \ref{lemma:beta-conv}, we see that $w.p.1$ after some $T_{\epsilon_s}$ time, $\|\bv{\beta}(t)\|\leq \xi$.
This implies that, for all $t\geq T_{\epsilon_s}$ and all $s_i$, we have $w.p.1$ that:
\begin{eqnarray}
g_{s_i}(\bv{\beta}(t))\leq Vf_{\max}+r\xi B, \quad \forall\,s_i.
\end{eqnarray}
Now note that $g(\bv{\beta}(t), t)$ can be expressed as:
\begin{eqnarray}
\hspace{-.3in}&&g(\bv{\beta}(t), t)= \sum_{s_i}[\pi_{s_i} +\delta_{s_i}(t)] g_{s_i}(\bv{\beta}(t)), \label{eq:dual-separate}
\end{eqnarray}
where $\delta_{s_i}(t) = \pi_{s_i}(t) - \pi_{s_i}$ denotes the error of the empirical  distribution.
Therefore,
\begin{eqnarray}
\hspace{-.4in}&&|g(\bv{\beta}(t), t) -  g(\bv{\beta}(t))| \leq \max_{s_i}|\delta_{s_i}(t)| \sum_{s_i}|g_{s_i}(\bv{\beta}(t))| \label{eq:dual-empirical}\\
\hspace{-.4in}&&\qquad\qquad\qquad\qquad \quad\stackrel{(a)}{\leq} \max_{s_i}|\delta_{s_i}(t)| M(Vf_{\max}+r\xi B). \nonumber
\end{eqnarray}
The  inequality (a) uses the fact that $g_{s_i}(\bv{\beta}(t))\leq Vf_{\max}+r\xi B$ for all $t\geq T_{\epsilon_s}$. This implies that with probability $1$, after some finite time $T_{\epsilon_s}$, (\ref{eq:dual-empirical}) holds for all points $\bv{\beta}(t)$  generated  by solving (\ref{eq:primal-approx}).
\end{proof}

\section*{Appendix C -- Proof of Lemma \ref{lemma:conv}}
Here we prove Lemma \ref{lemma:conv}. Since the type of convergence we consider here happens with probability $1$, in the following, we will sometimes  say ``converge'' for convenience when it is not confusing.
\begin{proof} (Lemma \ref{lemma:conv})
Note that for all $t$, the empirical dual function $g(\bv{\beta}, t)$ is always concave \cite{bertsekasoptbook} and hence continuous. 
Suppose $\bv{\beta}(t)$ does not converge to $\bv{\gamma}^*$. Then, there exists a $\delta>0$ such that we can  find an infinite sequence $\{\bv{\beta}(t_k)\}_{k=1}^{\infty}$ with $t_k\rightarrow\infty$ such that $\|\bv{\beta}(t_k) - \bv{\gamma}^*\|\geq \delta$ for all $k$.
Since $g(\bv{\gamma})$ has a unique optimal, this implies that there exists a constant $\epsilon_\delta>0$ such that for all $k$,
\begin{eqnarray}
g(\bv{\gamma}^*)-g(\bv{\beta}(t_k))\geq\epsilon_{\delta}. \label{eq:g-value-ineq}
\end{eqnarray}
To see why (\ref{eq:g-value-ineq}) holds, suppose this is not the case. Then, for any small $\epsilon>0$, we can find some $\bv{\beta}(t_k)$ such that (\ref{eq:g-value-ineq}) is violated. This implies that there is a sequence of $\{g(\bv{\beta}(t_m))\}_{m=1}^{\infty}$ which converges to $g(\bv{\gamma}^*)$. Denote $\script{M} = \{ \bv{\beta}: \|\bv{\beta}\|\leq \xi\}$. We see that $\script{M}$ is compact. Also, by Lemma  \ref{lemma:beta-conv},  we see that there exists a finite $m^*$, such that  $\{\bv{\beta}(t_m)\}_{m=m^*}^{\infty}$ is an infinite sequence in the compact set $\script{M}$ $w.p.1$. Hence, it has a converging subsequence \cite{math-rudin}. Denote the limit point of this subsequence by $\bv{\beta}'$. The above thus means that $g(\bv{\beta}') = g(\bv{\gamma}^*)$.  However, in this case  $\|\bv{\beta}' - \bv{\gamma}^*\|\geq\delta$, which contradicts the fact that $\bv{\gamma}^*$ is the unique optimal of $g(\bv{\gamma})$.

Now let us choose a time $T$ such that $w.p.1$, for all $t\geq T$,
\begin{eqnarray}
|g(\bv{\beta}, t) -  g(\bv{\beta})|\leq \epsilon_\delta/3,  \label{eq:g-value-time}
\end{eqnarray}
for all $\bv{\beta}\in\script{M}$.
This is always possible using Corollary \ref{corollary:uniform-conv}  and the fact that $\max_{s_i}|\delta_{s_i}(t)| \rightarrow0$ as $t\rightarrow\infty$. Then, choose a point $\bv{\beta}(t_k)$ from $\{\bv{\beta}(t_k)\}_{k=1}^{\infty}$ with $t_k\geq T$. We see that the optimal solution $\bv{\beta}(t_k)$ at time $t_k$ satisfies:
\begin{eqnarray}
g(\bv{\beta}(t_k), t_k)\geq  g(\bv{\gamma}^*, t_k). \label{eq:g-value-time-1}
\end{eqnarray}
Using (\ref{eq:g-value-time}),  (\ref{eq:g-value-time-1}) implies that:
\begin{eqnarray}
g(\bv{\beta}(t_k)) + \frac{2}{3}\epsilon_\delta\geq g(\bv{\gamma}^*),
\end{eqnarray}
which contradicts (\ref{eq:g-value-ineq}). This shows that $\bv{\beta}(t)$  converge to $\bv{\gamma}^*(t)$ $w.p.1$.
\end{proof}

\section*{Appendix D -- Proof of Lemma \ref{lemma:drift-g-tilde}}
Our proof is similar to the one in \cite{huangneely_dr_tac}, except that in this case, we prove the convergence result for a shifted target $\bv{\tilde{\gamma}^*}(t)$.
\begin{proof} (Lemma \ref{lemma:drift-g-tilde})
First, we  define a Lyapunov function:
\begin{eqnarray}
L(t) = \|\bv{q}(t) -\tilde{\bv{\gamma}}^*(t) \|^2.
\end{eqnarray}
Since with probability $1$, $\bv{\beta}(t)$ converges to $\bv{\gamma}^*$, we see that  for any $\epsilon>0$, with probability $1$, there exists a time $T_0<\infty$ such that $|\beta_j(t) - \gamma_{j}^*|\leq \sqrt{\epsilon/r}$ for all $j=1, ..., r$. Therefore, we can choose $\epsilon$ with $ \sqrt{\epsilon/r}\leq \theta_j$ for all $j$, so that $w.p.1$, for all $t\geq T_{0}$, we have $\tilde{\bv{\gamma}}^*(t)\succeq\bv{0}$.

Then, from the queueing dynamic equation (\ref{eq:queuedynamic}), we know that $\bv{q}(t+1)$ is obtained by projecting $\bv{q}(t)-\bv{\mu}(t)+\bv{A}(t)$ onto $\mathbb{R}_+^r$. Hence, we have:
\begin{eqnarray}
&&\|\bv{q}(t+1) -\tilde{\bv{\gamma}}^*(t) \|^2 \label{eq:single-state} \\
&\stackrel{(a)}{\leq}& \|  \bv{q}(t)-\bv{\mu}(t)+\bv{A}(t) - \tilde{\bv{\gamma}}^*(t)   \|^2\nonumber\\
&\leq& \|  \bv{q}(t) - \tilde{\bv{\gamma}}^*(t)   \|^2 + \|    \bv{\mu}(t) - \bv{A}(t)\|^2\nonumber\\
&&  -2(\tilde{\bv{\gamma}}^*(t)- \bv{q}(t))^T( \bv{A}(t))-\bv{\mu}(t)). \nonumber
\end{eqnarray}
Here (a) uses the non-expansion property of projection \cite{bertsekasoptbook}.
Since $\bv{\mu}(t)$ and $\bv{A}(t)$ are chosen using $\bv{q}(t)+\bv{\beta}(t)-\bv{\theta}$ as the weights in (\ref{eq:lac-eq}), comparing (\ref{eq:lac-eq}) and (\ref{eq:dual_separable-aug}), and using the augmented problem (\ref{eq:primal-aug}), we obtain:
\begin{eqnarray}
&&(\tilde{\bv{\gamma}}^*(t)- \bv{q}(t))^T( \bv{A}(t))-\bv{\mu}(t))\nonumber\\
&\stackrel{(a)}{\geq}&  \tilde{g}_{S(t)}(\tilde{\bv{\gamma}}^*(t))- \tilde{g}_{S(t)}(\bv{q}(t))\nonumber\\
&=&  g_{S(t)}(\bv{\gamma}^*)- g_{S(t)}(\bv{q}(t)+\bv{\beta}(t)-\bv{\theta}).\label{eq:single-state-2} 
\end{eqnarray}
Here in (a) we have used the fact that $\bv{A}(t)-\bv{\mu}(t)$ is a subgradient of $\tilde{g}_{S(t)}(\bv{\gamma})$ at $\bv{\gamma}=\bv{q}(t)$  \cite{bertsekasoptbook}.
Plugging (\ref{eq:single-state-2}) into (\ref{eq:single-state}) and taking expectations over $S(t)$ conditioning on $\bv{q}(t)$, we get:
\begin{eqnarray}
\hspace{-.3in}&& \expect{\|\bv{q}(t+1) -\tilde{\bv{\gamma}}^*(t) \|^2 \left.|\right. \bv{q}(t)}\label{eq:exp-0} \\
\hspace{-.3in}&&\quad\leq  \|  \bv{q}(t) - \tilde{\bv{\gamma}}^*(t)   \|^2 +B - 2(g(\bv{\gamma}^*)- g(\bv{q}(t)+\bv{\beta}(t)-\bv{\theta})). \nonumber
\end{eqnarray}
The constant $B$ is due to $ \|    \bv{\mu}(t) - \bv{A}(t)\|\leq B$.
Now for a given $\eta>0$, if:
\begin{eqnarray}
&&\eta^2-2\eta \| \bv{q}(t) - \tilde{\bv{\gamma}}^*(t) \| \label{eq:cond}\\
&&\qquad\geq B - 2(g(\bv{\gamma}^*)- g(\bv{q}(t)+\bv{\beta}(t)-\bv{\theta})), \nonumber
\end{eqnarray}
then we can plug this into (\ref{eq:exp-0}) and obtain:
\begin{eqnarray}
&& \expect{\|\bv{q}(t+1) -\tilde{\bv{\gamma}}^*(t) \|^2 \left.|\right. \bv{q}(t)}\label{eq:exp-1} \\
&\leq&  \|  \bv{q}(t) - \tilde{\bv{\gamma}}^*(t)   \|^2  - 2\eta \|  \bv{q}(t) - \tilde{\bv{\gamma}}^*(t)   \| +\eta^2 \nonumber\\
&=& (  \|  \bv{q}(t) - \tilde{\bv{\gamma}}^*(t)   \| - \eta )^2. \nonumber
\end{eqnarray}
This implies that:
\begin{eqnarray}
\hspace{-.4in}&& \expect{\|\bv{q}(t+1) -\tilde{\bv{\gamma}}^*(t) \| \left.|\right. \bv{q}(t)}
\leq    \|  \bv{q}(t) - \tilde{\bv{\gamma}}^*(t)   \| - \eta. \label{eq:drift-in-proof}
\end{eqnarray}
For (\ref{eq:drift-in-proof}) to hold,  we only need (\ref{eq:cond}) to hold. Rearranging the terms in (\ref{eq:cond}) and using the fact that $\tilde{\bv{\gamma}}^*(t)=\bv{\gamma}^*-\bv{\beta}(t)+\bv{\theta}$,  it becomes:
\begin{eqnarray*}
&&2(g(\bv{\gamma}^*)- g(\bv{q}(t)+\bv{\beta}(t)-\bv{\theta})) \\
&&\qquad \geq 2\eta \| \bv{q}(t)+\bv{\beta}(t)-\bv{\theta} -  \bv{\gamma}^*\| +B-\eta^2.
\end{eqnarray*}
This holds whenever:
\begin{eqnarray}
&&\rho||\bv{\gamma}^*- \bv{q}(t)- \bv{\beta}(t)+\bv{\theta}|| \label{eq:polyhedral-cond-ineq-0}\\
&&\qquad \geq \eta \| \bv{q}(t)+\bv{\beta}(t)-\bv{\theta} -  \bv{\gamma}^*\| +\frac{B-\eta^2}{2}.  \nonumber
\end{eqnarray}
By choosing $0<\eta<\rho$  and using (\ref{eq:polyhedral-cond-ineq-0}), we see that (\ref{eq:drift-in-proof}) holds whenever:
\begin{eqnarray*}
\|\bv{q}(t)  -\tilde{\bv{\gamma}}^*(t)\| = ||\bv{\gamma}^*- \bv{q}(t)- \bv{\beta}(t)+\bv{\theta}|| \geq D_{p}\triangleq\frac{B-\eta^2}{2(\rho -\eta)}.
\end{eqnarray*}
This proves the lemma.
\end{proof}

\section*{Appendix E -- Proof of Theorem \ref{theorem:q-poly}}
Here we prove Theorem \ref{theorem:q-poly}. Different from the proof in \cite{huangneely_dr_tac}, where $\bv{q}(t)$ is shown to be attracted to a fixed point,  here $\bv{q}(t)$ can be viewed to be chasing a moving target (by Lemma \ref{lemma:drift-g-tilde}).  
Fortunately, we see that $\tilde{\bv{\gamma}}^*(t)\rightarrow\bv{\theta}$ $w.p.1$. Hence, $\bv{q}(t)$ will eventually be attracted to $\bv{\theta}$.
\begin{proof} (Theorem \ref{theorem:q-poly}) First of all, we have the following inequalities:
\begin{eqnarray*}
\|\bv{q}(t) -  \tilde{\bv{\gamma}}^*(t)   \|\leq\|\bv{q}(t) -  \bv{\theta}  \| + \|\bv{\beta}(t) - \bv{\gamma}^*\|, \\
\|\bv{q}(t) -  \tilde{\bv{\gamma}}^*(t)   \|\geq\|\bv{q}(t) -  \bv{\theta}  \| -  \|\bv{\beta}(t) - \bv{\gamma}^*\|.
\end{eqnarray*}
Since with probability $1$, $\bv{\beta}(t)$ converges to $\bv{\gamma}^*$ (Lemma \ref{lemma:conv}), we see that  for any $\epsilon$, with probability $1$, there exists a time $T_{\epsilon}<\infty$ such that $\|\bv{\beta}(t) - \bv{\gamma}^*\|\leq \epsilon$ for all $t\geq T_{\epsilon}$. Using this fact in (\ref{eq:lac_exp_drift}),  
we see that $w.p.1$, when $t\geq T_{\epsilon}$,
\begin{eqnarray}
\hspace{-.0in}\expect{\|\bv{q}(t+1)  -\bv{\theta}\|\left|\right.\bv{q}(t)} \leq \|\bv{q}(t) - \bv{\theta}\|-\eta +2\epsilon.  \label{eq:lac_exp_drift-2}
\end{eqnarray}
Choosing $2\epsilon<\eta$ and defining $\eta_1=\eta-2\epsilon>0$, we see that for all time $t\geq T_\epsilon$, one has:
\begin{eqnarray}
\hspace{-.0in}\expect{\|\bv{q}(t+1)  -\bv{\theta}\|\left|\right.\bv{q}(t)} \leq \|\bv{q}(t) - \bv{\theta}\|-\eta_1. \label{eq:lac_exp_drift-3}
\end{eqnarray}
Note that this holds $w.p.1$ whenever $t\geq T_{\epsilon}$ and $\|\bv{q}(t)  -\tilde{\bv{\gamma}}^*(t)\|\geq D_p\triangleq \frac{B-\eta^2}{2(\rho-\eta)}$, which is satisfied whenever $t\geq T_{\epsilon}$  and $\|\bv{q}(t)  -\bv{\theta}\| \geq \tilde{D}_p\triangleq D_p + \epsilon$.

Having established (\ref{eq:lac_exp_drift-3}), we can now use an argument as in \cite{huangneely_dr_tac} and show that there exist constants $c_p=\Theta(1)$, $K_p=\Theta(1)$ such that $w.p.1$, \footnote{Note that this result holds for any $\bv{q}(0)<\infty$. Therefore, even though the $\Theta(1)$ drift in (\ref{eq:lac_exp_drift-3}) becomes effective only after some time, it does not affect the result. }
\begin{eqnarray}
\script{P}_p(\tilde{D}_p, m)\leq c_pe^{-K_pm},\label{eq:pm_ineq}
\end{eqnarray}
where $\script{P}_p(\tilde{D}_p, m)$ is defined as:
\begin{eqnarray}
\script{P}_p(\tilde{D}_p, m)\triangleq\limsup_{t\rightarrow\infty}\frac{1}{t}\sum_{\tau=0}^{t-1}\prob{\|\bv{q}(\tau)-\bv{\theta}\|>\tilde{D}_p+m}. \label{eq:pm_def}
\end{eqnarray}
This implies that for any $q_j(t)$, the probability that $q_j(t)>\theta_j+\tilde{D}_p+m$ is exponentially decreasing in $K_pm$ with $K_p=\Theta(1)$. Hence, if we define $d_j(t) = \max[q_i(t) - \theta_j, 0]$, it can be shown that $\overline{d}_j(t) = \Theta(1)$. Thus, we have:
\begin{eqnarray}
\overline{q}_{\text{av}}^{\mathtt{OLAC}} = \sum_j\overline{q}_j = \sum_j\theta_j +\Theta(1).
\end{eqnarray}
This completes the proof of Theorem \ref{theorem:q-poly}.
\end{proof}

\section*{Appendix F -- Proof of Theorem \ref{theorem:cost}}
To prove  Theorem \ref{theorem:cost}, we first have the following simple lemma, which will be used in the proof.
\begin{lemma}\label{lemma:avg-rate}
Suppose $g_0(\bv{\gamma})$ is polyhedral.
Let $\overline{\mu}_j^{\mathtt{OLAC}}$ and $\overline{A}_j^{\mathtt{OLAC}}$ be the average service rate and average arrival rate to queue $j$ under $\mathtt{OLAC}$. Then, if $\theta_j>\tilde{D}_p+\delta_{\max}$, we have:
\begin{eqnarray}
\overline{\mu}^{\mathtt{OLAC}}_j - \overline{A}^{\mathtt{OLAC}}_j\leq \delta_{\max}c_pe^{- K_p(\theta_j-\tilde{D}_p-\delta_{\max})}, \,\, w.p.1. \,\,\Diamond
\end{eqnarray}
\end{lemma}
\begin{proof} (Lemma \ref{lemma:avg-rate})
From (\ref{eq:pm_ineq}), we know that $w.p.1.$, if $\theta_j>\tilde{D}_p+\delta_{\max}$, then: \begin{eqnarray}
\limsup_{t\rightarrow\infty}\frac{1}{t}\sum_{\tau=0}^{t-1}\prob{q_j(\tau)<\delta_{\max}}\leq c_pe^{- K_p(\theta_j-\tilde{D}_p-\delta_{\max})}. \label{eq:pm_def_j}
\end{eqnarray}
This shows that the fraction of  time that $q_j(t)$ is smaller than $\delta_{\max}$ is at most $c_pe^{-K_p(\theta_j-\tilde{D}_p-\delta_{\max})}$. The lemma then follows since every  queue  $j$ can only serve $\delta_{\max}$ packets in a timeslot.
\end{proof}

We are now ready to prove Theorem \ref{theorem:cost}.
\begin{proof} (Theorem \ref{theorem:cost})
We define a Lyapunov function $L(t)=\frac{1}{2}\sum_jq^2_j(t)$ and the one-slot conditional drift $\Delta(t)\triangleq \expect{L(t+1) - L(t) \left.|\right. \bv{q}(t)}$. Using the queueing dynamic equations (\ref{eq:queuedynamic}), we have:
\begin{eqnarray}
\Delta(t) \leq B - \sum_{j}q_j(t) \expect{\mu_j(t) - A_j(t)\left.|\right. \bv{q}(t)}.
\end{eqnarray}
By adding to both sides the term $V\expect{f(t)\left.|\right. \bv{q}(t)} - \sum_{j}\expect{(\beta_j(t)-\theta_j) [\mu_j(t) - A_j(t)]\left.|\right. \bv{q}(t)}$, we obtain:
\begin{eqnarray}
\hspace{-.3in}&& \Delta(t) + V\expect{f(t)\left.|\right. \bv{q}(t)} \label{eq:drift-aug}\\
\hspace{-.3in}&& \qquad\,- \sum_{j}\expect{(\beta_j(t)-\theta_j) [\mu_j(t) - A_j(t)]\left.|\right. \bv{q}(t)} \nonumber\\
\hspace{-.3in}&&\qquad \leq B + V\expect{f(t)\left.|\right. \bv{q}(t)} \nonumber\\
\hspace{-.3in}&&\qquad\quad - \sum_{j}\expect{(q_j(t) +\beta_j(t) -\theta_j) [\mu_j(t) - A_j(t)]\left.|\right. \bv{q}(t)}. \nonumber
\end{eqnarray}
Using Assumption \ref{assumption:equal} and plugging in (\ref{eq:drift-aug}) the optimal stationary and randomized policy $\{x^{(s_i)*}_k, \vartheta^{(s_i)*}_k\}_{s_i, k}$, we  obtain:
 \begin{eqnarray}
&& \Delta(t) + V\expect{f^{\mathtt{OLAC}}(t)\left.|\right. \bv{q}(t)} \label{eq:drift-aug-1}\\
&& \qquad- \sum_{j} \expect{(\beta_j(t)-\theta_j) [\mu^{\mathtt{OLAC}}_j(t) - A^{\mathtt{OLAC}}_j(t)]\left.|\right. \bv{q}(t)} \nonumber\\
&&\qquad \leq B + Vf_{\text{av}}^*. \nonumber
\end{eqnarray}
Here $f^{\mathtt{OLAC}}(t)$, $\mu^{\mathtt{OLAC}}_j(t)$ and $A^{\mathtt{OLAC}}_j(t)$ denote the cost, service rate and arrival rate under the $\mathtt{OLAC}$ algorithm.
Rearranging the terms, we have:
 \begin{eqnarray}
&& \Delta(t) + V\expect{f^{\mathtt{OLAC}}(t)\left.|\right. \bv{q}(t)}  \leq B + Vf_{\text{av}}^*\label{eq:drift-aug-2}\\
&&\qquad\, + \sum_{j}\expect{(\beta_j(t)-\theta_j) [\mu^{\mathtt{OLAC}}_j(t) - A^{\mathtt{OLAC}}_j(t) ]\left.|\right. \bv{q}(t)}. \nonumber
\end{eqnarray}
Taking expectations on both sides of  (\ref{eq:drift-aug-2}) over $\bv{q}(t)$,
 taking a telescoping sum over $t=0, ..., T-1$, rearranging the terms, dividing both sides by $TV$, and taking a $\limsup$ as $T\rightarrow\infty$, we have:
 \begin{eqnarray}
\hspace{-.3in}&&\limsup_{T\rightarrow\infty}\frac{1}{T} \sum_{t=0}^{T-1}\expect{f^{\mathtt{OLAC}}(t)}  \leq \frac{B}{V} + f_{\text{av}}^*\label{eq:drift-aug-3}\\
\hspace{-.3in}&&\quad +\limsup_{T\rightarrow\infty} \frac{1}{T} \sum_{t=0}^{T-1}\expect{\sum_{j}\frac{(\beta_j(t)-\theta_j)}{V} [\mu^{\mathtt{OLAC}}_j(t) - A^{\mathtt{OLAC}}_j(t)]}. \nonumber
\end{eqnarray}
It remains to show that the last term is $O(1/V)$. From Lemma \ref{lemma:conv} we know that $\beta_j(t)-\theta_j$ converges to $\gamma_{j}^*-\theta_j=\Theta(V)$ with probability $1$.  Thus, $w.p.1$, there exists a finite time $T_{\epsilon}$, such that for all $t\geq T_{\epsilon}$,
\begin{eqnarray}
|(\beta_j(t)-\theta_j)  - (\gamma_{j}^*-\theta_j) |\leq \epsilon, \,\,\forall\, j. \end{eqnarray}
Hence, we have:
\begin{eqnarray}
\hspace{-.3in}&&\limsup_{T\rightarrow\infty}\frac{1}{T} \sum_{t=0}^{T-1}\expect{\frac{(\beta_j(t)-\theta_j)}{V}  [\mu^{\mathtt{OLAC}}_j(t) - A^{\mathtt{OLAC}}_j(t)]} \nonumber\\
\hspace{-.3in}&&\stackrel{(a)}{\leq}  (\frac{\gamma_{j}^*-\theta_j  +\epsilon}{V})\limsup_{T\rightarrow\infty}\frac{1}{T}\sum_{t=0}^{T-1}  \expect{\mu^{\mathtt{OLAC}}_j(t)} \nonumber\\
\hspace{-.3in}&&\qquad -   (\frac{\gamma_{j}^*-\theta_j  -\epsilon }{V})\limsup_{T\rightarrow\infty}\frac{1}{T}\sum_{t=0}^{T-1}  \expect{A^{\mathtt{OLAC}}_j(t)} \nonumber\\
\hspace{-.3in}&& \leq \frac{\gamma_{j}^*-\theta_j}{V}(\overline{\mu}^{\mathtt{OLAC}}_j - \overline{A}^{\mathtt{OLAC}}_j) +\frac{2\epsilon\delta_{\max}}{V}. \label{eq:last-term}
\end{eqnarray}
Here (a) follows from the fact that $T_{\epsilon}<\infty$.
Using Lemma \ref{lemma:avg-rate} with $\theta_j=[\log(V)]^2$ and a sufficiently large $V$ such that $K_p([\log(V)]^2 - \tilde{D}_p - \delta_{\max})\geq2 \log(V)$, we have:
\begin{eqnarray}
\overline{\mu}^{\mathtt{OLAC}}_j - \overline{A}^{\mathtt{OLAC}}_j\leq \frac{\delta_{\max} c_pe^{K_p(\tilde{D}_p+\delta_{\max}})}{ e^{K_p[\log(V)]^2}} = O(\frac{1}{V^2}). \label{eq:last-term-2}
\end{eqnarray}
We first consider when $\gamma_{0j}^*>0$. In this case, $\gamma_j^* = V\gamma_{0j}^*=\Theta(V)$. Thus, (\ref{eq:last-term-2}) implies that:
 \begin{eqnarray}
 (\gamma_{j}^*-\theta_j)(\overline{\mu}^{\mathtt{OLAC}}_j - \overline{A}^{\mathtt{OLAC}}_j)  = O(\frac{1}{V}). \label{eq:last-term-3}
\end{eqnarray}
Using (\ref{eq:last-term}) and  (\ref{eq:last-term-2}), we conclude that, with probability $1$,
 \begin{eqnarray}
\hspace{-.5in}&&\limsup_{T\rightarrow\infty}\frac{1}{T} \sum_{t=0}^{T-1}\expect{f^{\mathtt{OLAC}}(t)}  \leq  f_{\text{av}}^*+\frac{B}{V} +O(\frac{1}{V}). \label{eq:drift-aug-4}
\end{eqnarray}
In the case when $\gamma_j^*=0$, we see that (\ref{eq:last-term-3}) still holds. Hence, (\ref{eq:drift-aug-4}) again follows from (\ref{eq:last-term-3}).
This completes the proof of Theorem \ref{theorem:cost}.
\end{proof}

\section*{Appendix G -- Proof of Theorem \ref{theorem:convergence-bp}}
To prove the theorem, we make use of the following technical lemma from \cite{bertsekasoptbook}.
\begin{lemma}\label{lemma:exp-time}
Let $\script{F}_n$ be filtration, i.e., a sequence of increasing $\sigma$-algebras with $\script{F}_n\subset\script{F}_{n+1}$. Suppose the sequence of random variables $\{y_n\}_{n\geq0}$ satisfy:
\begin{eqnarray}
\expect{||y_{n+1}-y^*||\left.|\right. \script{F}_n} \leq \expect{||y_{n}-y^*|| \left.|\right. \script{F}_n} -u_n, \label{eq:exp-dec}
\end{eqnarray}
where $u_n$ takes the following values:
\begin{eqnarray}
u_n=\left\{\begin{array}{ll} u & \textrm{if  $||y_{n}-y^*||\geq D$}, \\0 &\textrm{else}.\end{array}\right.\label{eq:eta-n}
\end{eqnarray}
Here $u>0$ is a given constant.
 Then, by defining $N_{D}\triangleq\inf\{k \left.|\right. \|y_{n}-y^*\| \leq D\}$, we have:
\begin{eqnarray}
\expect{N_D} \leq ||y_{0}-y^*|| /u. \,\, \Diamond
\end{eqnarray}
\end{lemma}
\begin{proof}
See \cite{bertsekasoptbook}.
\end{proof}

\begin{proof} (Theorem \ref{theorem:convergence-bp})
From Lemma \ref{lemma:drift-g-tilde}, we see that (\ref{eq:lac_exp_drift})  holds for all $t$ under Backpressure with $\eta<\rho$ and $D_p=\frac{B-\eta^2}{2(\rho -\eta)}$. Hence, using Lemma \ref{lemma:exp-time}, we have:
\begin{eqnarray*}
\expect{T_{D_p}}\leq ||\bv{q}(0)-\bv{\gamma}^*||/\eta.
\end{eqnarray*}
On the other hand, since $\|\bv{q}(t+1) -\bv{q}(t)\|\leq B$, i.e., in every time the queue vector can change by at most $B$ distance, we must have
$\expect{T_{D_p}}\geq (||\bv{q}(0)-\bv{\gamma}^*|| -D_p)^+ /B$.
\end{proof}

\section*{Appendix H -- Proof of Theorem \ref{theorem:lac2-convergence}}
We prove Theorem \ref{theorem:lac2-convergence} in this section. We first have the following lemma regarding the distance between $\bv{\beta}(t)$ and the distribution estimation accuracy.
\begin{lemma}\label{lemma:beta-rate}
With probability $1$, there exists an $\Theta(1)$ time $T_{\epsilon_s}$ (defined in Lemma \ref{lemma:beta-conv}) such that, for all $t\geq T_{\epsilon_s}$,
\begin{eqnarray}
\| \bv{\beta}(t) - \bv{\gamma}^*\|\leq 2 \max_{s_i}|\delta_{s_i}(t)|M (Vf_{\max}+r\xi B)/\rho. \label{eq:beta-rate}
\end{eqnarray}
\end{lemma}
\begin{proof}
See Appendix I.
\end{proof}
 Lemma \ref{lemma:beta-rate} shows that under the polyhedral condition, the distance between the current estimate of the Lagrange multiplier value and the true value diminishes as the empirical distribution converges. This is an important result, as it allows us to focus mainly on the convergence of the distribution when studying the algorithm's convergence.
To show our results, we also  make use of the following theorem regarding distribution convergence \cite{chung_concentration}.
\begin{theorem}\label{theorem:concentration}
Let $X_1$, ..., $X_n$ be independent random variables with $\prob{X_i=1}=p_i$, and $\prob{X_i=0}=1-p_i$. Consider $X=\sum_{i=1}^nX_i$ with expectation $\expect{X}=\sum_{i=1}^np_i$. Then, we have:
\begin{eqnarray}
\prob{X\leq \expect{X}-m} &\leq& e^{\frac{-m^2}{2\expectm{X}}}, \label{eq:low-tail}\\
\prob{X\geq \expect{X}+m} &\leq& e^{\frac{-m^2}{2(\expectm{X}  +m/3 )}}. \quad\Diamond\label{eq:high-tail}
\end{eqnarray}
\end{theorem}
We now state the proof of the convergence time for $\mathtt{OLAC2}$.

\begin{proof} (Theorem \ref{theorem:lac2-convergence}) Choosing $n=T_l=V^{c}$ and $m=4V^{c/2}\log(V)$ in Theorem \ref{theorem:concentration}, and let $X_t$ be the indicator of state $s_i$ at time $t$ for $0\leq t\leq n-1$. Then, we have $\expect{X}=V^c\pi_{s_i}$. Thus, using the fact that $\pi_{s_i}\leq1$, we see that at time $t=T_l$,
\begin{eqnarray*}
 \prob{\pi_{s_i}(t)\leq \pi_{s_i} - \frac{4\log(V)}{V^{c/2}}  } &\leq& e^{-8[\log(V)]^2}, \label{eq:low-tail}\\
 \prob{\pi_{s_i}(t)\geq  \pi_{s_i} + \frac{4\log(V)}{V^{c/2}} } &\leq& e^{-\frac{8[\log(V)]^2}{ 1 + \frac{2}{3}\log(V) V^{-c/2} }}. 
\end{eqnarray*}
Using  the union bound,  we see that when $V$ is large enough such that $\frac{2}{3}\log(V)V^{-c/2}\leq1$,
\begin{eqnarray}
\prob{\max_{s_i}|\pi_{s_i}(t) - \pi_{s_i}| \geq \frac{4\log(V)}{V^{c/2}} } \leq Me^{-4[\log(V)]^2}.
\end{eqnarray}
Therefore, using Lemma \ref{lemma:beta-rate}, we see that with probability at least $1- \frac{M}{V^{4\log(V)}}$, at time $t=T_l$,
\begin{eqnarray}
\| \tilde{\bv{\beta}} - \bv{\gamma}^*\| &\leq& \frac{ 8\log(V) M (Vf_{\max}+r\xi B) }{\rho V^{c/2}} \nonumber\\
&=& \Theta(V^{1-c/2}\log(V)),
\end{eqnarray}
which implies that at time $t=T_l$,
\begin{eqnarray}
\| \bv{q}(T_l) - \bv{\gamma}^*\| = \Theta(V^{1-c/2}\log(V)).
\end{eqnarray}
Now note that given the same backlog values,  $\mathtt{OLAC2}$ always uses the same actions as LIFO-Backpressure. Using Lemma \ref{lemma:drift-g-tilde} with $T_{\epsilon_s}=0$, we see that at any time $t\geq T_l$, if $\| \bv{q}(t) - \bv{\gamma}^*\|\geq D_p$,  we have:
\begin{eqnarray}
 \expect{\|\bv{q}(t+1) -\bv{\gamma}^* \| \left.|\right. \bv{q}(t)}\leq \|  \bv{q}(t) - \bv{\gamma}^*   \| - \eta. \label{eq:lac2-drift}
\end{eqnarray}
Using Lemma \ref{lemma:exp-time}, we conclude that with probability at least $1- \frac{M}{V^{4\log(V)}}$:
\begin{eqnarray}
\expect{T_{D_p}} &=& O(V^{1-c/2}\log(V)) +T_l\\
&=& O(V^{1-c/2}\log(V) + V^c).
\end{eqnarray}
This completes the proof of Theorem \ref{theorem:lac2-convergence}.
\end{proof}

\section*{Appendix I -- Proof of Lemma \ref{lemma:beta-rate}}
Here we prove Lemma \ref{lemma:beta-rate}.
\begin{proof} (Lemma \ref{lemma:beta-rate})
As in Lemma \ref{lemma:beta-conv}, we have:
\begin{eqnarray}
\hspace{-.3in}&&g(\bv{\beta}, t)= \sum_{s_i}[\pi_{s_i} +\delta_{s_i}(t)] g_{s_i}(\bv{\beta}). \label{eq:dual-separate-0}
\end{eqnarray}
Also, for any $t$,  $g(\bv{\beta}(t), t)\geq g(\bv{\gamma}^*, t)$. By Lemma \ref{lemma:beta-conv}, we see that $w.p.1$,  for all $t\geq T_{\epsilon_s}$,
\begin{eqnarray*}
\hspace{-.3in}&&|g(\bv{\beta}(t), t) -  g(\bv{\beta}(t))| \leq  \max_{s_i}|\delta_{s_i}(t)| M(Vf_{\max}+r\xi B).
\end{eqnarray*}
This implies that:
\begin{eqnarray}
\hspace{-.4in}&&g(\bv{\gamma}^*)-g(\bv{\beta}(t)) \leq 2 \max_{s_i}|\delta_{s_i}(t)| M(Vf_{\max}+r\xi B). \label{eq:distance-rate}
\end{eqnarray}
This is so because if (\ref{eq:distance-rate}) does not hold, then:
\begin{eqnarray*}
\hspace{-.3in}&&g(\bv{\beta}(t), t)- g(\bv{\gamma}^*, t)\\
\hspace{-.3in}&\leq& g(\bv{\beta}(t)) + \max_{s_i}|\delta_{s_i}(t)| M (Vf_{\max}+r\xi B) \\
\hspace{-.3in}&&- \big[g(\bv{\gamma}^*) -  \max_{s_i}|\delta_{s_i}(t)| M (sVf_{\max}+r\xi B)\big]\\
\hspace{-.3in}&<&0, 
\end{eqnarray*}
which contradicts with $g(\bv{\beta}(t), t)\geq g(\bv{\gamma}^*, t)$. 
Now with (\ref{eq:distance-rate}) and (\ref{eq:polyhedral}), i.e., $g(\bv{\gamma}^*) - g(\bv{\beta})\geq \rho \| \bv{\gamma}^* - \bv{\beta}  \|$,  we see that $w.p.1$, for all $t\geq T_{\epsilon_s}=\Theta(1)$,
\begin{eqnarray}
\| \bv{\beta}(t) - \bv{\gamma}^*\|\leq 2 \max_{s_i}|\delta_{s_i}(t)| M(Vf_{\max}+r\xi B)/\rho.
\end{eqnarray}
This completes the proof of  Lemma \ref{lemma:beta-rate}.
\end{proof}


$\vspace{-.2in}$

\end{document}